\documentclass[a4paper,11pt]{amsart}
\usepackage{amssymb}
\usepackage{cmap}
\usepackage{hyperxmp}
\usepackage[pdfdisplaydoctitle=true,
            colorlinks=true,
            urlcolor=blue,
            citecolor=blue,
            linkcolor=blue,
            pdfstartview=FitH,
            pdfpagemode= UseNone,
            bookmarksnumbered=true]{hyperref}
\usepackage{todonotes}
\usepackage[all]{xy}

\theoremstyle{plain}
\newtheorem{thm}{Theorem}[section]
\newtheorem*{thm*}{Theorem}

\newtheorem{cor}[thm]{Corollary}
\newtheorem{lem}[thm]{Lemma}
\newtheorem{prop}[thm]{Proposition}

\theoremstyle{definition}
\newtheorem{defn}[thm]{Definition}

\theoremstyle{definition}
\newtheorem{rem}[thm]{Remark}

\newtheorem{expl}[thm]{Example}
\newtheorem*{ass*}{Assumption}
\newtheorem{qn}[thm]{Question}

\numberwithin{equation}{section}

\newcommand{\RR}{\mathbb{R}} 

\newcommand{\calR}{\mathcal{R}} 
\newcommand{\II}{\textup{II}}

\newcommand{\id}{\mathrm{id}}           
\let\on=\operatorname

\newcommand{\Imm}{\mathrm{Imm}}         

\newcommand{\ud}{\,\mathrm{d}}

\def\th{\theta}

\def\Ga{\Gamma}

\def\Ph{\Phi}

\let\on=\operatorname

\def\Imm{\on{Imm}}

\allowdisplaybreaks

\newcommand{\I}{\mathcal{I}}
\newcommand{\N}{\mathcal{N}}

\newcommand{\e}{\varepsilon}
\newcommand{\calH}{\mathcal{H}}

\newcommand{\inj}{\operatorname{inj}}
\newcommand{\Hol}{\operatorname{Hol}}
\newcommand{\dist}{\operatorname{dist}}

\newcommand{\pl}{\partial}

\newcommand{\brk}[1]{\left(#1\right)}          
\newcommand{\Brk}[1]{\left[#1\right]}          
\newcommand{\BRK}[1]{\left\{#1\right\}}        
\newcommand{\Abs}[1]{\left|#1\right|}        
\newcommand{\inner}[1]{\left\langle#1\right\rangle}      

\hypersetup{
    pdftitle={Sobolev metrics on spaces of manifold valued curves}, 
    pdfsubject={MSC 2010: 58B20, 58D10, 35G55, 35G60}, 
    pdfkeywords={Sobolev metrics, manifold valued curves}, 
    pdflang=EN 
    }
\author{Martin Bauer}
\address{Martin Bauer: Department of Mathematics, Florida State University}
\email{bauer@math.fsu.edu}
\author{Cy Maor}
\address{Cy Maor: Einstein Institute of Mathematics, The Hebrew University of Jerusalem}
\email{cy.maor@mail.huji.ac.il}
\author{Peter W. Michor}
\address{Peter W. Michor: Faculty for Mathematics, University of Vienna}
\email{peter.michor@univie.ac.at}
\thanks{M. Bauer was partially supported by NSF-grant 1912037 (collaborative research in connection with NSF-grant 1912030). 
		C. Maor was partially supported by ISF-grant 1269/19.}

\begin{document}

\title{Sobolev metrics on spaces of manifold valued curves}

\date{}

\begin{abstract}
We study completeness properties of reparametrization invariant Sobolev metrics of order $n\ge 2$ on the space of manifold valued open and closed immersed curves.
In particular, for several important classes of metrics, we show that Sobolev immersions are metrically and geodesically complete (thus the geodesic equation is globally well-posed).
These results were previously known only for closed curves with values in Euclidean space.
For the class of constant coefficient Sobolev metrics on open curves, we show that they are metrically incomplete, and that this incompleteness  only arises from curves that vanish completely (unlike ``local" failures that occur in lower order metrics). 
\end{abstract}

\maketitle

\setcounter{tocdepth}{1}
\tableofcontents

\section{Introduction and main results}

\subsection{Background}
In recent years Riemannian geometry on the space of curves has been an area of active research. The motivation for these investigations can be found in the area of shape analysis, where the space of geometric curves plays an important role: closed planar curves are used to encode the outlines (shapes) of planar objects, and elastic (reparametrization invariant) Riemannian metrics have been successfully used to compare these objects in a variety of different applications \cite{srivastava2010shape,srivastava2016functional,younes1998computable,younes2010shapes}. More recently, curves with values in a manifold have emerged as a topic of interest in shape analysis as well. Examples include the study of trajectories on the earth~\cite{su2014statistical,su2018comparing}, of computer animations~\cite{celledoni2016shape}, or of brain connectivity data~\cite{dai2019analyzing}. Here the brain connectivity of a patient over time is represented as a path in the space of positive, symmetric matrices. Motivated by these applications several of the numerical algorithms, as originally developed for planar curves, have been generalized to this more complicated situation. 

In this article we are interested in the mathematical properties of these Riemannian metrics and in particular in questions related to completeness of the corresponding geodesic equations. These investigations build up on classical questions related to diffeomorphism groups, as reparametrization invariant metrics on spaces of immersions can be viewed as generalizations of right-invariant metrics on diffeomorphism groups. 
These have been in the focus of intense research due to their relations to many prominent PDEs via Arnold's approach to hydrodynamics~\cite{arnold1966geometrie,arnold1999topological, vizman2008geodesic}. 
Local well-posedness in this setup was established for a wide variety of invariant metrics, typically using an Ebin--Marsden-type analysis~\cite{ebin1970groups,misiolek2010fredholm,kolev2017local,michor2007overview,bauer2014overview}.
The focus of this article is geodesic and metric completeness, which is well understood for strong enough metrics in the case of diffeomorphism groups \cite{younes2010shapes,mumford2013euler,misiolek2010fredholm,bruveris2017completeness,bauer2020well}, but is mostly open for spaces of immersions.
For closed, regular curves with values in Euclidean space, a series of completeness results both on the space of parametrized and unparametrized curves has been obtained, beginning with Bruveris, Michor and Mumford~\cite{bruveris2014geodesic}, see also 
\cite{bruveris2015completeness,bruveris2017completeness2,bauer2015metrics}. 
The goal of this article is to generalize these results to the case of open and closed, regular curves with values in a Riemannian manifold. While the manifold structure of the target space is of little relevance for the local results mentioned before, it significantly complicates the analysis for the global results studied in the present article. We will comment on the differences with the Euclidean situation in Section~\ref{sec:structure} below; first we describe the main contributions of the present article.

\subsection{Main Result}\label{sec:main_result}
To formulate our main result we first introduce the manifold of regular curves and the class of Riemannian metrics that we will consider in this article. For $n\geq 2$, we consider the space of Sobolev immersions from a one-dimensional parameter space $D$ with values in a complete  Riemannian manifold with bounded geometry $(\N,g)$:
 \begin{equation}
\mathcal I^n(D,\N)=\left\{c\in H^n(D,\N): c'(\theta)\neq 0,\; \forall \theta \in D \right\}.
\end{equation}
Here $D=[0,2\pi]$ for open curves and $D=S^1$ for closed curves. 
The Sobolev space $H^n(D,\N)$ is defined in more detail in Section~\ref{sec:spaces}; note that $H^n(D,\N)\subset C^1(D,\N)$, hence the condition $c'(\theta)\ne 0$ is well defined.
On this space we can consider reparametrization invariant (elastic) Sobolev metrics of order $n$. The   class we focus on in this paper is given by
\begin{equation}\label{def:SobMetric}
\begin{split}
&G_c(h,k) =  \sum_{i=0}^n a_i(\ell_c) \int_{D} g(\nabla_{\pl_s}^i h, \nabla_{\pl_s}^i k)\ud s,
\end{split}
\end{equation}
where $a_i\in C^\infty((0,\infty),[0,\infty))$, 
$\nabla$ is the covariant derivative in $\N$, and $s=|c'|$ is the norm of $c'$ with respect to the Riemannian metric $g$. Furthermore,  $\ud s=|c'|\ud\theta$ is the arc length one form, $\pl_s = \frac{1}{|c'|}\pl_\theta$ is the arc length vector field along the curve, and $\ell_c=\int_D  \ud s$ is the length of the curve.
The two most important sub-families of these type are:
\begin{enumerate}
\item The \emph{constant coefficient} Sobolev metrics, where $a_i(\ell_c)=C_i\ge 0$ are constants and do not depend on the length $\ell_c$;
\item The family of \emph{scale invariant} Sobolev metrics where $a_i(\ell_c)= C_i \ell_c^{2n-3}$ with $C_i\ge 0$ being again constants. In this case, when the target manifold $\N$ is the Euclidean space, composition with rescaling $x\mapsto \alpha x$ of the target manifold is an isometry of $(\I^n(D,\N), G)$, for each $\alpha>0$.
\end{enumerate}
In both cases we assume that $C_0$ and $C_n$ are strictly positive, to avoid degeneracy.
The main focus of the present article lies on completeness properties of these Riemannian metrics. In a slightly simplified version our main results can be summarized as follows:
\begin{thm*}[Main Theorem]
Let $D=[0,2\pi]$ or $D=S^1$, and let $G$  be the scale invariant Sobolev metric~\eqref{def:SobMetric} of order $n\geq 2$. The following completeness properties hold: 
\begin{enumerate}
\item $(\I^n(D,\N),\dist^G)$ is a complete metric space.
\item $(\I^n(D,\N),G)$ is geodesically complete
\item Any two immersions in the same connected component of $(\I^n(D,\N),G)$ can be joined by a minimizing geodesic.
\end{enumerate}
For $D=S^1$ the results continue to hold for the family of constant coefficient Sobolev metrics.
\end{thm*}
Previously this result was only known for closed curves in Euclidean space (see \cite{bruveris2015completeness} for constant coefficients and \cite{bruveris2017completeness} for a wider class that includes scale invariant ones), and thus the results of the present article generalize these previous works in two important directions (open curves and curves with values in a manifold). 
In fact, we will prove these statements for a wider class of metrics, see Theorems~\ref{thm:metric_completeness}--\ref{thm:metric_completeness2}.
Note that in this infinite dimensional situation the theorem of Hopf-Rinow is not valid~\cite{atkin1975hopf} and thus item (3) does not follow directly from the metric completeness, but has to be proven separately.

\subsection{Further contributions of the article}
In the following we describe several further key contributions of the current article:
\begin{itemize}
\item
{\bf Completeness in the smooth setting:}
In the main theorem above, we have formulated the results only in the Sobolev category. Using an Ebin--Marsden-type no-loss-no-gain result~\cite{ebin1970groups}, we show that geodesic completeness (i.e., global existence of geodesics) extends to the space of smooth, closed curves (Corollary~\ref{cor:smooth}).
For open curves, we only obtain regularity in the interior of the curve, as explained in Section~\ref{sec:smooth}.

\item {\bf Metric incompleteness of constant coefficient metrics on open curves:}
In \cite{bauer2019relaxed} it was observed that the space of open curves, with respect to constant coefficient Sobolev metrics, is metrically incomplete; 
indeed, in the same paper the authors constructed a path of immersed curves, whose lengths tend to zero after finite time.
In Section~\ref{sec:completion_open_curves} we elaborate on this example, and show that vanishing of the entire curve is the only way a path (or a sequence) of immersed curves can leave the space of immersions $\I^n([0,2\pi],\N)$ in finite time (Theorem~\ref{thm:metric_completion}).
That is, a path cannot leave the space by some "local" failure, say, by losing the immersion property at a point (such a failure of completeness can occur in lower-order metrics, e.g., in shockwaves in the inviscid Burgers equation).
We give some evidence that the completion of the space of open curves in this case is a one-point completion, where the additional point represents all the Cauchy sequences converging to vanishing length curves.

\item {\bf Existence of minimizing geodesics for  constant coefficient metrics on open curves:}
We show that if the distance between two open curves is lower than some explicit threshold depending only on their lengths, then they can be connected by a minimizing geodesic (Theorem~\ref{thm:bdry_prlm}).
We note, however, that this threshold is not necessarily sharp; in fact, in view of the rather rigid way in which curves can leave the space, we cannot rule out that a minimizing geodesic exists between \emph{any} two immersions.
We also do not know whether \emph{geodesics} (unlike general paths of finite length) may cease to exist after finite time, that is, we do not know if the space is geodesically incomplete (only that it is metrically incomplete).
These questions will be considered in future works.

\item 
{\bf Local well-posedness:} 
Our completeness results are only valid for metrics of order $n\geq 2$, and it can be shown that metrics of lower order can never have these properties. 
Nevertheless, using an Ebin--Marsden-type approach, we show local well-posedness for all smooth metrics of the type \eqref{def:SobMetric} of order $n\geq 1$, see Theorem~\ref{local_wellposedness}.
This result was previously known for closed curves and the case of open curves requires some additional considerations for dealing with the boundary terms that appear in the geodesic equation.

\item {\bf  Completeness of the intrinsic metric on $H^n(D,\N)$:}
It is well known that $H^n(D,\N)$, for $n > \dim D/2$, is a Hilbert manifold, and that its topology coincides with the one induced via the inclusion $H^n(D,\N)\subset H^n(D,\RR^m)$ that is defined by a closed isometric embedding $\iota:\N\to \RR^m$.
This inclusion also induces a complete metric space structure on $H^n(D,\N)$.
As part of the proof of the main theorem, we show that the natural Riemannian metric on $H^n(D,\N)$,
\begin{equation}\label{eq:Hn_metric}
\calH_c(h,k) :=  \int_{D} g_c(h, k) + g_c(\nabla_{\partial_\theta}^n h, \nabla_{\partial_\theta}^n k)\ud \theta,
\end{equation}
is also metrically complete (Proposition~\ref{prop:Hn_Hilbert_manifold}), thus defining a complete metric space structure that is intrinsic (independent of an isometric embedding).
We study these different definitions and their equivalence in Section~\ref{sec:spaces}.

\end{itemize}

\subsection{Main ideas in the proof and structure of the article}\label{sec:structure}
The techniques used in the proof of our main theorem, Theorems~\ref{thm:metric_completeness}--\ref{thm:metric_completeness2},  expand upon the ones used to study completeness of Euclidean curves~\cite{bruveris2015completeness}. The main difficulties arise from taking into account the more complicated structure of the space $\mathcal I^n(D, \mathcal N)$ and the effects of the curvature  of $\mathcal N$ on various estimates (in particular, on the behavior of some Sobolev interpolation inequalities). To give the reader a first glimpse, we will outline the strategy and main steps below.

\paragraph{{\bf Local well-posedness}}
As a basis to the rest of the analysis, we first study the metric $G$ (as in \eqref{def:SobMetric}) in Section~\ref{sec:sobolev_metrics}, and prove that it is a smooth, strong metric on $\I^n(D,\N)$.\footnote{A Riemannian metric $G$ on a manifold $\mathcal{M}$ is a section of non-degenerate bilinear forms on the tangent bundle. 
A \emph{strong} Riemmanian metric also satisfies that for each $x\in \mathcal{M}$, the topology induced by $G_x$ on $T_x\mathcal{M}$ coincides with the original topology (induced by the manifold structure) on $T_x\mathcal{M}$. 
If $\dim\mathcal{M}<\infty$, every metric is a strong one, but in infinite dimensions this is not the case.}
In this section we also give some details on the associated geodesic equation and formulate the local well-posedness result (as this theorem is not the focus of the present article, we postpone its proof to Appendix~\ref{sec:local_well_posed}).

\paragraph{{\bf Metric and geodesic completeness}}
The space of Sobolev immersions $\I^n(D,\N)$ is an open subset of $H^n(D,\N)$, which is metrically complete with respect to the metric $\calH$, defined in \eqref{eq:Hn_metric}; this is established in Section~\ref{sec:spaces}. Note that for $\N=\RR^d$ this is trivial, as $H^n(D,\N)$ is a Hilbert space in this case.

Since $(H^n(D,\N), \dist^\calH)$ is a complete metric space, showing metric completeness of $(\I^n(D,\N),\dist^G)$ can be reduced to showing that $G$ and $\calH$ are  equivalent metrics, uniformly on every $\dist^G$-ball in $\I^n(D,\N)$, and that the speed $|c'|$ of an immersion $c\in \I^n(D,\N)$ is bounded away from zero on every $\dist^G$-ball.
This reduction is done in detail in Section~\ref{sec:reductions}.

In order to obtain the uniform equivalence of $G$ and $\calH$ on metric balls, one needs to obtain bounds on the length $\ell_c$ of the curve, and on certain norms of the velocity $c'$, uniformly for all immersions $c$ in a metric ball.
This is done in Section~\ref{sec:estimates_length_speed}, and the proof of metric completeness is then concluded in Sections~\ref{sec:pf_metric_comp}--\ref{sec:pf_metric_comp2}.
As metric completeness implies geodesic completeness of strong Riemannian metrics also in infinite dimensions, see~\cite[VIII, Proposition~6.5]{lang2012fundamentals}, this also concludes the proof of geodesic completeness.

The main technical tool for establishing the bounds on $\ell_c$ and $c'$ are Sobolev interpolation inequalities on the tangent space $T_c\I^n(D,\N)$, with explicit dependence of the inequalities constants on the length of the base curve $c$.
In the case of closed curves, there is non-trivial holonomy along the curves, hence we need to control the holonomy along a curve in terms of its length, and apply these estimates to the interpolation inequalities (this is one of the main technical differences from the Euclidean case).
These are done in Section~\ref{sec:estimates}, though some of the geometric estimates are postponed to Appendix~\ref{app:estimates}.

\paragraph{{\bf Existence of minimal geodesics}}
To prove existence of minimal geodesics between two immersions $c_0$ and $c_1$, we consider the energy of paths $c_t:[0,1]\to \I^n(D,\N)$ between $c_0$ and $c_1$ (defined by the metric $G$), and use the direct methods of the calculus of variations to prove that a minimizing sequence of paths converges, in an appropriate sense, to an energy minimizer (which is, by definition, a geodesic). This is done in Section~\ref{sec:boundary_value_pblm}.
Since this approach relies heavily on weak convergence of paths, and the weak topology is not readily available on the Hilbert manifold $\I^n(D,\N)$, we first embed it into the Hilbert space $H^n(D,\RR^m)$ via a closed isometric embedding $\iota:\N\to \RR^m$.
The analysis then combines the same type of bounds that are used to prove metric completeness, with bounds that relate the metric on $H^n(D,\RR^m)$ to the metric $G$  on $\I^n(D,\N)$ (similar bounds are also used in proving the completeness of $H^n(D,\N)$ with respect to the $\dist^\calH$ metric in Section~\ref{sec:spaces}).

\subsection*{Acknowledgements}
We would like to thank to Martins Bruveris, FX Vialard and Amitai Yuval for various discussions during the work on this paper.

\section{Spaces of manifold valued functions and immersions}\label{sec:spaces}
Let $(\N,g)$ be a (possibly non-compact) complete Riemannian manifold with bounded geometry, 
where the induced norm of the Riemannian metric  will be denoted by $|\cdot|=\sqrt{g(\cdot,\cdot)}$.
We will denote its covariant derivative by $\nabla$, or, where ambiguity might arise, by $\nabla^\N$.
With a slight abuse of notation, we will also use it as the covariant derivative on pullbacks of $T\N$.

We consider the space
 of (closed or open) regular curves with values in $\N$, which we denote by
 \begin{equation}
\operatorname{Imm}(D,\N)=\left\{c\in C^{\infty}(D,\N):c'(\theta)\neq 0,\; \forall \theta \in D \right\}.
\end{equation}
Here $D=S^1$ for closed curves and $D=[0,2\pi]$ for open curves.
This space is an infinite dimensional manifold, whose tangent space at a curve $c$ is the space of vector fields along $c$:
 \begin{equation}
T_c\operatorname{Imm}(D,\N)=\left\{h\in C^{\infty}(D,T\N): \pi(h)=c \right\},
\end{equation}
where $\pi$ denotes the foot point projection from $T\N$ to $\N$.

To obtain the desired completeness and well-posedness results we need to consider a larger space of \textbf{Sobolev immersions} $\I^n(D,\N)\supset \Imm(D,\N)$, for $n\ge 2$, which we define below.

\begin{defn}\label{def:spaces}
Let $\N$ be a Riemannian manifold as above, and fix a proper, smooth, isometric embedding $\iota: \N \to \RR^m$, for large enough $m\in \mathbb{N}$.
For $n\ge 2$, we define the Sobolev space $H^n(D,\N)$ and the space of Sobolev immersions $\I^n(D,\N)$ as follows:
\begin{enumerate}
\item $H^n(D,\N)$ consists of all maps $c:D\to \N$ such that $\iota \circ c \in H^n(D; \RR^m)$.
\item $\I^n(D,\N)$ consists of all $c\in H^n(D,\N)$ such that $c'(\theta)\ne 0,\, \forall \theta \in D$.
\end{enumerate}
\end{defn}

With this (extrinsic) definition of $H^n(D,\N)$, it inherits the metric structure of $H^n(D;\RR^m)$, which we denote by $\dist^{\text{ext}}$; since convergence in the space $H^n(D;\RR^m)$ implies uniform convergence, we have that $H^n(D,\N)$ is a closed subset of $H^n(D;\RR^m)$, hence a complete metric space with respect to $\dist^{\text{ext}}$.
We are interested in characterizing $H^n(D,\N)$ as an infinite dimensional Riemannian manifold.
The main goal of this section is to prove the following:

\begin{prop}\label{prop:Hn_Hilbert_manifold}
The space $H^n(D,\N)$, $2\le n\in \mathbb N$ is a Hilbert manifold whose tangent space at $c$ is $H^n(D;c^*T\N)$.
Moreover, it is a complete metric space with respect to the distance function $\dist^{\calH}$ induced by the smooth Riemannian metric \eqref{eq:Hn_metric}:
\[
\calH_c(h,k) :=  \int_{D} g_c(h, k) + g_c(\nabla_{\partial_\theta}^n h, \nabla_{\partial_\theta}^n k)\ud \theta.
\]
Finally, the space of Sobolev immersions $\I^n(D,\N)$ is an open subset of $H^n(D,\N)$ and in particular is a Hilbert manifold with the same tangent space.
\end{prop}
Henceforth, we will always endow $H^n(D,\N)$ with the metric $\dist^{\calH}$ (rather than $\dist^{\text{ext}}$). 
Note that, in general, $\dist^{\calH}$ and $\dist^{\text{ext}}$ need not to be equivalent metrics.
Note also that for $h\in T_c\I^n(D,\N)$ there are two natural $L^2$ metrics: in one we integrate with respect to $\ud \theta$, and in the other with respect to arc length $\ud s = |c'|\ud\theta$; we denote the first one by $L^2(\ud \theta)$ and the second by $L^2(\ud s)$.

Proposition~\ref{prop:Hn_Hilbert_manifold} holds for much more general manifold domain $D$: namely, the Hilbert manifold structure exists whenever $2n > \dim D$, and the openness of $\I^n(D,\N)$ in $H^n(D,\N)$ holds whenever $2(n-1) > \dim D$.
These are known results and we describe their proofs below for completeness.
To the best of our knowledge, the completeness of $(H^n(D,\N), \dist^{\calH})$ has not been considered before; we expect it to hold, again, whenever $2n> \dim D$, virtually with the same proof as the one below (using H\"older inequalities instead of uniform bounds).

We start by proving a technical lemma that shows the local equivalence of the $\calH_c$ norm and the restriction of the standard $H^n(D;\RR^m)$ norm.
This lemma will be used both in the proof of Proposition~\ref{prop:Hn_Hilbert_manifold}, and also later, when we prove existence of minimizing geodesics between immersions in Section~\ref{sec:boundary_value_pblm}.

\begin{lem}\label{lem:extrinsic_vs_intrinsic_norms}
Let $\iota:\N \to \RR^m$ be an isometric embedding, and let $n\ge 2$.
Let $K\subset\N$ be a compact set, and let $c\in H^n(D,\N)$ be a curve whose image lies in $K$.
Let $C>0$ be such that
\[
\|\nabla_{\pl_\theta}^k c'\|_{L^2(\ud\theta)} < C, \qquad k=0,\ldots,n-1.
\]
For $h\in H^n(D;c^*T\N)$, denote by $\iota_* h \in H^n(D;\RR^m)$ the image of $h$ under the embedding.
The extrinsic norm of $h$ is then defined by
\[
\| h\|_{H^n(\iota)}^2 :=\|\iota_* h\|_{H^n(D;\RR^m)}^2 = \int_0^{2\pi} |\iota_*h|^2 + |\pl_\theta^n \iota_*h|^2\,\ud\theta,
\]
where $|\cdot|$ is the norm in $\RR^m$, and $\pl_\theta = \nabla^{\RR^N}_{c'}$ is the standard derivative on $\RR^m$.
Then, there exists a constant $\beta>0$, depending only on  $\iota$, $K$ and $C$ such that for every $h\in H^n(D;c^*T\N)$,
\[
\beta^{-1} \| h\|_{H^n(\iota)} \le \| h\|_{\calH_c} \le \beta \| h\|_{H^n(\iota)}.
\]
\end{lem}

Throughout the proof we will use standard Sobolev embedding results of the space $H^{n-1}(D,c^*T\N)$; these estimates can be found in any standard book on Sobolev spaces, e.g., \cite{leoni2017first}, and the adaptation from real-valued functions to vector-bundle-valued functions is straightforward using parallel transport along the curve to a single tangent space.
For completion, the estimates and their reduction to the real-valued case appear in Lemma~\ref{lem:high_order_ineq} below.

\begin{proof}
First, note that $c'\in H^{n-1}(D,c^*T\N)$, so the fact that there exists a bound on the $L^2$ norms of $\nabla_{\pl_\theta}^k c'$ for $k=0,\ldots, n-1$ is not an additional assumption. 
Also, by standard Sobolev estimates, $H^{k-1}(D,c^*T\N)$ continuously embeds into $C^{k-2}(D,c^*T\N)$, that is 
\[
\|\cdot \|_{C^{k-2}(D,c^*T\N)} \le C_{k,n,\dim \N}\|\cdot\|_{H^{k-1}(D,c^*T\N)},
\]
where the constant $C_{k,n,\dim \N}$ depends only on $k,n$ and the dimension of $\N$.
Therefore, we have that our bounds on $\|\nabla_{\pl_\theta}^k c'\|_{L^2(\ud\theta)}$ imply that
\[
\|\nabla_{\pl_\theta}^k c'\|_\infty < C, \qquad k=0,\ldots,n-2,
\]
by possibly enlarging the constant $C$.

Next, note that $ \| h\|_{L^2(\iota)} =  \| h\|_{L^2(\ud\theta)}$, since $\iota$ is an isometric embedding $|\iota_* h| = |h|$ pointwise for every $\theta$ (here, the $\RR^m$-norm appears on the left-hand side, the $T\N$-norm on the right-hand side).

Denote by $\II$ the second fundamental form of $\N$ in $\RR^m$, that is, for $v,w\in T_x\N$, we have
\[
\II(v,w) = \nabla^{\RR^m}_v w - \nabla^\N_v w.
\]
In a coordinate patch on a tubular neighborhood of $\N$, with coordinates $(x_i)_{i=1}^m$ such that $(x_a)_{a=1}^{d}$, where $d=\dim \N$ are coordinates on $\N$ and $\pl_{x_{\alpha}}\perp \pl_{x_a}$ for $a=1,\ldots,d$ and $\alpha = d+1,\ldots, m$, we have
\[
\II(v,w) = \Gamma_{ab}^\alpha(x) v^a w^b \pl_\alpha,
\] 
where $\Gamma_{ij}^k$ are the Christoffel symbols of $\nabla^{\RR^m}$ in these coordinates.
Since $\nabla^\N_v w \perp \II(v,w)$, we have
\begin{equation}\label{eq:ext_vs_int_der}
\begin{split}
|\pl_\theta \iota_* h|^2 &= |\nabla^\N_{\pl_\theta} h|^2 + |\II(c',h)|^2 \le  |\nabla^\N_{\pl_\theta} h|^2 + C^2 |\II|^2 |h|^2 \\
	&\le  |\nabla^\N_{\pl_\theta} h|^2 + C' |h|^2, 
\end{split}
\end{equation}
where $C'=C^2 \sup_{x\in K} |\II|^2$.
Integrating, we obtain
\[
\|\nabla^\N_{\pl_\theta} h\|_{L^2(\ud\theta)}^2 \le \|\pl_\theta \iota_* h\|_{L^2(\ud\theta)}^2 \le \|\nabla^\N_{\pl_\theta} h\|_{L^2(\ud\theta)}^2 + C' \|h\|_{L^2(\ud\theta)}^2 \lesssim \|h\|_{H^1(\ud\theta)}^2.
\]
Here and in the following we use the notation $\lesssim$ to indicate that there exists a constant, which does not depend on $h$, such that the inequality holds.
For the second order terms we calculate
\begin{equation}\label{eq:ext_vs_int_aux1}
\begin{split}
\pl_\theta^2 \iota_* h &= \pl_\theta \nabla^\N_{\pl_\theta} h + \pl_\theta (\II(c',h)) 
	= (\nabla^\N_{\pl_\theta})^2 h + \II(c',\nabla^\N_{\pl_\theta} h) + \pl_\theta (\II(c',h)),
\end{split}
\end{equation}
Since $\II$ and its derivatives are bounded on the compact set $K$, we have
\[
\begin{split}
|\pl_\theta^2 \iota_* h| 
	&\lesssim |(\nabla^\N_{\pl_\theta})^2 h| + |c'||\nabla^\N_{\pl_\theta} h| + |c'||h| + |\pl_\theta c'||h| + |c'| |\pl_\theta h| \\
	&\lesssim |(\nabla^\N_{\pl_\theta})^2 h| + |c'||\nabla^\N_{\pl_\theta} h| + |c'|(1+|c'|)|h| + |\nabla^\N_{\pl_\theta} c'||h|.
\end{split}
\]
where we used \eqref{eq:ext_vs_int_der} when changing $\pl_\theta$ to $\nabla^\N_{\pl_\theta}$ (applied to $c'$ and $h$).
Since $n\ge2$, we use again the Sobolev embedding $H^{n-1}(D,c^*T\N)\subset C^{n-2}(D,c^*T\N)\subset C^{0}(D,c^*T\N)$ to obtain that $|c'|<C$ and $\|h\|_{L^\infty} \le C_2 \|h\|_{\calH_c}$ for some $C_2>0$ depending only on the dimension.
We therefore have
\[
|\pl_\theta^2 \iota_* h| \lesssim |(\nabla^\N_{\pl_\theta})^2 h| + |\nabla^\N_{\pl_\theta} h| + |h| + \|h\|_{\calH_c}|\nabla^\N_{\pl_\theta} c'|.
\]
Squaring and integrating, and using that $\|\nabla^\N_{\pl_\theta} c'\|_{L^2} < C$, we obtain that,
\[
\|\pl_\theta^2 \iota_* h\|_{L^2} \lesssim \|(\nabla^\N_{\pl_\theta})^2 h\|_{L^2} + \|\nabla^\N_{\pl_\theta} h\|_{L^2} + \|h\|_{L^2} +  \|h\|_{\calH_c} \lesssim \|h\|_{\calH_c},
\]
and therefore 
\[
\| h\|_{H^{2}(\iota)} \lesssim \| h\|_{\calH_c}.
\] 
The converse inequality follows in a similar manner, by using \eqref{eq:ext_vs_int_aux1}, to bound $|(\nabla^\N_{\pl_\theta})^2 h|$ with $|\pl_\theta^2 \iota_* h|$ and lower order terms.

For $n>2$ the proof proceeds inductively in the same way --- writing $\pl_\theta^n \iota_* h$ in terms of $(\nabla^\N_{\pl_\theta})^n h$ and lower order terms that involve the second fundamental form and its derivatives (as in \eqref{eq:ext_vs_int_aux1}), and bounding the lower order terms in a similar manner.
\end{proof}

\noindent\emph{Proof of Proposition~\ref{prop:Hn_Hilbert_manifold}.}
\noindent\textbf{Part I: Smooth structure and topology.} 
An alternative characterization of $H^n(D,\N)$ is
\[
\begin{split}
H^n(D,\N) &= \Big\{c\in C(D,\N) \,:\, c=\exp_s (V)
	\\&\qquad\qquad \text{ for some } s\in C^\infty(D,\N),\, V\in H^n(D; s^*T\N)\Big\},
\end{split}
\]
where $\exp$ is the exponential map with respect to the Riemannian metric $g$ on $\N$ (see, e.g., \cite[Lemma B.5]{Weh04}).
This characterization induces a smooth structure on $H^n(D,\N)$, where the charts, modeled on $H^n(D; s^*T\N)$, are given by $\exp_s$ for $s\in C^\infty(D,\N)$. 
The tangent space at $c$ is $H^n(D;c^*T\N)$.
See \cite[5.3--5.8]{Michor20} for details.
This smooth structure is described in detail in \cite[Section~3]{IKT13} (it is denoted there by $\mathcal{A}^s_g$). In \cite[Proposition~3.7]{IKT13} it is shown that this smooth structure coincides with the one induced by considering local charts on $D$ and $\N$ (which provides yet another characterization to $H^n(D,\N)$).

Next, note that the topology induced by this smooth structure is equivalent to the topology induced on $H^n(D,\N)$ by $\dist^{\text{ext}}$  \cite[Lemma B.7]{Weh04}.
The inner product $\calH_c$ describes the Hilbert space topology on the tangent space 
$T_c H^n(D,\N)=H^n(D;c^*T\N)$. 
Since these are also the modeling spaces for the natural chart construction, $\calH$ is a strong Riemannian metric. 
Thus the distance function $\dist^{\calH}$ induced by $\calH$ induces the topology of $H^n(D,\N)$.

\noindent\textbf{Part II: Openness of $\I^n(D,\N)$ in $H^n(D,\N)$.}
Taking again the extrinsic point of view $\I^n(D,\N)$ is the intersection of $H^n(D,\N)$ with the set of maps $c\in H^n(D;\RR^m)$ such that $c'\ne 0$.
By the Sobolev embedding $\|c'\|_{L^\infty(D;\RR^m)} \le C \|c\|_{H^n(D;\RR^m)}$, which holds since $n\ge 2$, it is immediate that $c'\ne 0$ is an open condition in $H^n(D;\RR^m)$, and hence $\I^n(D,\N)$ is open in $H^n(D,\N)$.

\noindent\textbf{Part III: Completeness of $(H^n(D,\N), \dist^{\calH})$.}
Let $c_j\in H^n(D,\N)$ be a Cauchy sequence with respect to $\dist^{\calH}$.
We aim to show that $c_j$ is also a Cauchy sequence with respect to $\dist^{\text{ext}}$.
Then, since $(H^n(D,\N), \dist^{\text{ext}})$ is complete, we will obtain that the sequence converges to some $c_\infty\in H^n(D,\N)$; since the topologies induced by $\dist^{\text{ext}}$ and $\dist^{\calH}$ coincide, we will obtain that $(H^n(D,\N), \dist^{\calH})$ is complete as well.

Since $(c_j)_{j\in\mathbb{N}}$ is a $\dist^{\calH}$-Cauchy sequence, it lies inside some $\dist^{\calH}$-ball $B$ of radius $r>0$ centered at some $c_0\in H^n(D,\N)$. 
By taking a slightly larger $r$ we can also assume that for every $j\le k\in \mathcal{N}$ there exists a path $c_{jk}:[0,1]\to H^n(D,\N)$ connecting $c_j$ and $c_k$ (that is, $c_{jk}(0)=c_j$ and $c_{jk}(1) = c_k$), such that $c_{jk}(t)\in B$ for every $t\in [0,1]$ and $L^{\calH}(c_{jk}) < \dist^{\calH}(c_j,c_k) + \frac{1}{j}$, where $L^{\calH}$ is the length of $c_{jk}$ with respect to the metric $\calH$.

We now show that all the curves in $B$ lie inside a compact subset of $\N$; moreover, we show that for some $C>0$, all curves $c\in B$ satisfy 
\[
\|\nabla_{\pl_\theta}^k c'\|_{L^2(d\th)} < C, \qquad k=0,\ldots,n-1.
\]
It then follows by Lemma~\ref{lem:extrinsic_vs_intrinsic_norms} that there exists a constant $\beta>0$, such that for every $c\in B$ and every $h\in H^n(D;c^*T\N)$,
\[
\|\iota_*h\|_{H^n(D;\RR^m)} \le \beta \| h\|_{\calH_c},
\]
where $\iota_* h \in H^n(D;\RR^m)$  is the image of $h$ under the embedding, and where $\|\cdot\|_{H^n(D;\RR^m)}$ is the standard norm in $H^n(D;\RR^m)$ (see Lemma~\ref{lem:extrinsic_vs_intrinsic_norms}).
Therefore, for every $j\le k\in \mathcal{N}$,
\[
\dist^{\text{ext}}(c_j,c_k) \le L^{\text{ext}}(c_{jk}) \le \beta L^{\calH}(c_{jk}) < \beta\brk{\dist^{\calH}(c_j,c_k) + \frac{1}{j}},
\]
where $L^{\text{ext}}$ is the length with respect to the external structure. 
Thus $(c_j)_{j\in\mathbb{N}}$ is a $\dist^{\text{ext}}$-Cauchy sequence and the proof is complete.

It remains to verify the assumptions of Lemma~\ref{lem:extrinsic_vs_intrinsic_norms}.
Let now $\bar{c}\in B=B(c_0,r)$.
By definition, there exists a path $c:[0,1]\to H^n(D,\N)$, with $c(0)=c_0$ and $c(1) = \bar{c}$ such that $L^\calH(c) < r$.
Now, for every $\theta_0 \in D$, we have
\[
\begin{split}
\dist_\N(c_0(\theta_0),\bar{c}(\theta_0)) &\le \int_0^1 |\pl_t c(t,\theta_0)|\,\ud t \le \int_0^1 \|\pl_t c\|_{L^\infty} \le C\int_0^1 \|\pl_t c_t\|_{\calH_c} \\
	&= CL^\calH(c) < Cr,
\end{split}
\]
where we used the Sobolev embedding on vector bundles $H^{n-1}(D,c^*T\N)\subset C^{n-2}(D,c^*T\N)\subset C^{0}(D,c^*T\N)$ as in the proof of Lemma~\ref{lem:extrinsic_vs_intrinsic_norms}.
It follows that the images of all the curves in $B$ lie in a compact subset of $\N$ (namely a neighborhood of radius $Cr$ around the image of $c_0$).

Now, let $k=0,\ldots,n-1$, then
\[
\begin{split}
\|\nabla_{\pl_\theta}^k \bar{c}'\|_{L^2(\ud\th)} - \|\nabla_{\pl_\theta}^k c_0'\|_{L^2(\ud\th)}
	&= \int_0^1 \pl_t \brk{\int_D |\nabla_{\pl_\theta}^k c'|^2\,\ud\theta }^{1/2} \,\ud t \\
	&= \int_0^1 \frac{\int_D g(\nabla_{\pl_\theta}^k c',\nabla_{\pl_\theta}^k \pl_t c') \,\ud\theta}{\brk{\int_D |\nabla_{\pl_\theta}^k c'|^2\,\ud\theta }^{1/2}}\,\ud t \\
	&\le \int_0^1 \brk{\int_D |\nabla_{\pl_\theta}^k \pl_t c'|^2\,\ud\theta }^{1/2}\, \ud t \\
	&\le \int_0^1 \| \pl_t c\|_{\calH_c} \,\ud t = L^\calH(c) < r,
\end{split}
\]
where we used again that the $L^2$ norms of $\nabla_{\pl_\theta}^{k} h$ for $k=0,\ldots, n$ are controlled by $\|h\|_{\calH_c}$ (again, we refer to Lemma~\ref{lem:high_order_ineq} for an exact statement).
The uniform bound on $\|\nabla_{\pl_\theta}^k \bar{c}'\|_{L^2(\ud\th)}$ immediately follows, and thus the assumptions of Lemma~\ref{lem:extrinsic_vs_intrinsic_norms} are fulfilled, uniformly on $B$. \hfill \qedsymbol{}

\section{Reparametrization invariant Sobolev metrics on spaces of  curves}\label{sec:sobolev_metrics}
\subsection{The metric and geodesic equation in the smooth category}
As detailed in the introduction, we are interested in reparametrization invariant 
Sobolev metrics on the spaces $\Imm(D,\N)$ and $\I^n(D,\N)$ defined above, and, more accurately, in metrics of the type \eqref{def:SobMetric}:
\begin{equation*}
\begin{split}
&G_c(h,k) =  \sum_{i=0}^n a_i(\ell_c) \int_{D} g(\nabla_{\pl_s}^i h, \nabla_{\pl_s}^i k)\ud s,\\
&a_i\in C^\infty((0,\infty),[0,\infty)), \text{ for } i=0,\ldots,n \quad\text{and}\quad a_0, a_n >0.
\end{split}
\end{equation*}

We now calculate the geodesic equation associated with $G_c$ in smooth settings; in the next subsection we extend the treatment to Sobolev settings.
To derive the geodesic equation it will be more convenient to write the metric using the so-called inertia operator, i.e., 
use integration by parts to write $G$ as
\begin{equation}
G_c(h,k)=\int_{D}g(A_c h,k)\ud s + B_c(h,k).
\end{equation}
Here 
\begin{equation}\label{eq:inertiaoperator}
A_c: T_c\operatorname{Imm}(D,\N)\to T_c\operatorname{Imm}(D,\N),
\end{equation}
is called the inertia operator of the metric $G$ and 
$B_c(h,k)$ depends solely on the boundary of $D$ and stems from the integration by parts process. 
Thus for closed curves the operator $B$ is not present.

\begin{lem}
The inertia operator of the metric \eqref{def:SobMetric} takes the form:
\begin{equation}\label{eq:inertiaoperator_closed}
A_c(h)=\sum_{i=0}^n (-1)^i a_i(\ell_c) \nabla_{\pl_s}^{2i}h,
\end{equation}
For open curves, i.e. $D=[0,2\pi]$, the boundary operator $B$ is given by:
\begin{equation}\label{eq:inertiaoperator_open}
B_c(h,k)= \sum_{i=1}^n a_i(\ell_c)\sum_{j=0}^{i-1} (-1)^{i+j-1}g(\nabla_{\pl_s}^{i+j}h,\nabla_{\pl_s}^{i-j-1}k)\Big|^{2\pi}_0\;.
\end{equation}
\end{lem}
\begin{proof}
These formulas follow directly from the integration by parts formula
\begin{equation}
\int_D g(h,\nabla_{\pl_s} k) \ud s=g(h,k)|_{\partial D} -\int_D g(\nabla_{\pl_s} h, k) \ud s\;.
\end{equation}
Note that for closed curves we have $D=S^1$ and thus $\partial D=\emptyset$.
 \end{proof}
Before we calculate the geodesic equation we will collect variational formulas of several quantities 
that appear in the metric.
In the following we will denote the variation of a quantity in direction $h\in T_c\Imm(D,\N)$ by $D_{c,h}$.
\begin{lem}\label{lem:variationformulas}
Let $c\in \Imm(D,\N)$ and $h\in T_c \Imm(D,\N)$. Then
\begin{align}
D_{c,h} |c'| &=  g(v, \nabla_{\pl_s} h) |c'| \label{eq:speed_derivative} \\
D_{c,h} \ud s&= g(v, \nabla_{\pl_s} h) \ud s\\
D_{c,h}\ell_c &= \int_D g(v, \nabla_{\pl_s} h) \ud s \label{eq:ell_derivative}
\end{align}
where $v=c'/{|c'|}$ denotes the unit length tangent vector to the curve $c$.
Extending the connection, as described in \cite[Section 3]{bauer2011sobolev}, we can also calculate the variation of the covariant derivtive $\nabla_{\pl_s}$ applied to a tangent vector $k\in T_c\Imm(D,\N)$:
\begin{align}
\label{eq:noncommuting_cov_der}
 \nabla_{h} \nabla_{\pl_s} k &= -g(v, \nabla_{\pl_s} h)  \nabla_{\pl_s} k + \nabla_{\pl_s} \nabla_h k+\mathcal R(v,h)k;
\end{align}
where $\mathcal R$ denotes the (Riemannian) curvature of $(\mathcal N, g)$.
\end{lem}
\begin{proof}
The first three formulas follow by straight-forward calculations, similar as for curves with values in Euclidean spaces, see, e.g., \cite{michor2007overview,bruveris2015completeness}.
For the last formula we follow the more general presentation in \cite{bauer2011sobolev}, where the variation of the Laplacian for $D$ being a compact manifold of arbitrary dimension has been derived. 
Using the formula
\begin{align}
\nabla_{h}\nabla_{\partial_{\theta}} k =   \nabla_{\partial_{\theta}} \nabla_{h}k +\mathcal R(h, c')k 
\end{align}
for swapping covariant derivatives, see e.g. \cite[Section 3.8]{bauer2011sobolev}, we obtain
\begin{align*}
\nabla_{h} \nabla_{\pl_s} k  &= D_{f,h}\left(|c'|^{-1}\right)\nabla_{\partial_\theta}k+ |c'|^{-1}\nabla_{h}\nabla_{\partial_\theta}k\\
&=-g(\nabla_{\pl_s} h,v) \nabla_{\pl_s} k+ |c'|^{-1}\nabla_{\partial_\theta}\nabla_{h} k+|c'|^{-1}\mathcal R(h, c')k
\end{align*}
which concludes the proof since $v=|c'|^{-1}c'$. 
\end{proof}

We are now able to calculate the geodesic equation. In the following calculation we will restrict to first order metrics, for which the exact form of the geodesic spray will be of importance in the proof of the local well-posedness result. 
For higher order metrics the existence and well-posedness of the geodesic equation will follow from general principles on strong metrics and we will thus not 
include these cumbersome calculations. 
The interested reader can consult the related calculations in \cite{bauer2011sobolev}, where the geodesic equations are derived for general higher order metrics (under the assumption that $D$ has no boundary).
The geodesic equation for constant coefficient metrics on closed curves in Euclidean space also appears in \cite[Theorem~1.1]{bruveris2014geodesic}.

\begin{lem}\label{thm:geodesicequation}
The geodesic equation of the first-order Sobolev-type metric, as defined in~\eqref{def:SobMetric} for $n=1$, is given by the set of equations:
\begin{multline*}
\nabla_{\partial_t} (A_c c_t)= -g(v,\nabla_{\pl_s} c_t) A_c c_t-\frac12\Psi_c(c_t,c_t)\nabla_{\pl_s} v-g(\nabla_{\pl_s} c_t,A_c c_t) v \\
		+ a_1(\ell_c)\mathcal R(c_t,\nabla_{\pl_s} c_t)v,
\end{multline*}
where the quadratic form $\Psi_c(c_t,c_t)$ is given by
 \begin{align*}
 \Psi_c(c_t,c_t)
 	&= a_0(\ell_c)g( c_t, c_t)+ a_0'(\ell_c)\int_D  g(c_t,c_t) \ud s \\
 	&\quad-a_1(\ell_c)g(\nabla_{\pl_s} c_t,\nabla_{\pl_s} c_t)+a_1'(\ell_c) \int_D g(\nabla_{\pl_s} c_t,\nabla_{\pl_s} c_t)\ud s
\end{align*}
For open curves, $D=[0,2\pi]$, we get the following boundary conditions:
\begin{align*}
\Big(-2\nabla_{\partial_t}\brk{a_1(\ell_c) \nabla_{\pl_s} c_t}+\Psi_c(c_t,c_t)v\Big)\bigg|_{\theta = 0,2\pi}=0\;.
\end{align*}
\end{lem}
The proof of this result is postponed to Appendix~\ref{app:proof_geodesicequation}.

\subsection{The induced metric on Sobolev immersions}
To obtain the desired completeness and well-posedness results we consider the extension of the metric $G$ (of order $n$) on the Banach manifolds $\I^q(D,\N)\supset \Imm(D,\N)$, for $q\ge \max\{n,2\}$, as defined in Definition~\ref{def:spaces} above.

Our aim in the rest of the section is to show the smoothness of the metrics $G$ on $\I^q(D,\N)$ (assuming $q\ge n$).
First, we need to introduce some mixed order spaces:
\begin{defn}
Let $q\ge 2$ and $q\ge k\geq0$. We define the function space:
 $$H^{k}_{\I^q}(D,T\N)=\left\{h\in H^k(D,T\N): \pi\circ h\in \I^q(D,\N)\right\}\;.$$
\end{defn}
We have the following result concerning their manifold structure and the operator $\nabla_{\pl_\theta}$:
\begin{lem}\label{lem:Hnk}
The spaces $H^{k}_{\I^q}(D,T\N)$ are smooth Hilbert manifolds
for any $q\ge 2$ and $q\ge k\ge 0$. The mapping 
\begin{equation}
\nabla_{\pl_\theta}: H^{k}_{\I^q}(D,T\N)\to H^{k-1}_{\I^q}(D,T\N)
\end{equation}
is a bounded linear mapping for $1\le k\le q$.
\end{lem}
\begin{proof}
The first part of this result can be found in \cite [Theorem~2.4]{bauer2020fractional}, while the second part follows directly from the definition of the 
space $H^{k}_{\I^q}(D,T\N)$.
\end{proof}
Note that $H^{q}_{\I^q}(D,T\N)= T\I^q(D,\N)$. 
If $k<q$ then $H^{k}_{\I^q}(D,T\N)$ is the the robust fiber completion of the weak Riemannian manifold $(\I^q(D,T\N), G^k)$ with the Sobolev metric $G^k$ from \eqref{def:SobMetric} in the sense described in \cite{Micheli2013}. These spaces will appear, when we repeatedly apply $\nabla_{\pl_s}$ to a vector field $h$ along an $H^n$-immersion ($\nabla_{\pl_s}$ will reduce the order of the vector field, but not of its foot point).
To show the smoothness of the metric we need the following result:
\begin{lem}\label{lem:smoothness_nablas}
Let $q\ge2$. Then the mapping 
\begin{align}
H^{k+1}_{\I^q}(D,T\N)&\rightarrow H^{k}_{\I^q}(D,T\N)\\
h&\mapsto \nabla_{\pl_s} h=\frac{1}{|\pi(h)|}\nabla_{\pl_\theta}h
\end{align}
is smooth for any $k\geq 0$.
\end{lem}
\begin{proof}
The mapping
\begin{align}
H^{k+1}_{\I^q}(D,T\N)&\rightarrow H^{k}_{\I^q}(D,T\N)\\
h&\mapsto \nabla_{\pl_\theta} h
\end{align}
is smooth by Lemma~\ref{lem:Hnk}. 
By the module properties of Sobolev spaces,
multiplication $H^{q}(D,\mathbb R)\times H^{k}_{\I^q}(D,T\N) \rightarrow H^{k}_{\I^q}(D,T\N)$
is smooth for $q\geq 2$ and $k\geq 0$.
Thus the result follows since $|\pi(h)|\in H^{q}(D,\mathbb R)$.
\end{proof}
Using this lemma  we immediately obtain the smoothness of the metric:
\begin{thm}\label{thm:metric_sobcompletion}
Let $q \ge 2$. 
Consider the Sobolev metric $G$ on $\Imm(D,\N)$ of order $n\leq q$ of the form \eqref{def:SobMetric}. Then $G$ extends to 
a smooth Riemannian metric on $\I^q(D,\N)$. For $q=n$ the metric $G$ is a strong Riemannian metric on $\I^n(D,\N)$.
\end{thm}
\begin{proof}
Iterating Lemma \ref{lem:smoothness_nablas} we have that 
\begin{equation}
\nabla_{\pl_s}^i:  T\I^q(D,\N)=H^{q}_{\I^q}(D,T\N)\rightarrow H^{q-i}_{\I^q}(D,T\N)\subset L^{2}_{\I^q}(D,T\N)
\end{equation}
is smooth for $0\leq i\leq n$. Thus 
the mapping 
\begin{align*}
T\I^q(D,\N)\times_{\I^q}T\I^q(D,\N)&\to L^{1}(D,\mathbb R)\\
(h,k)&\mapsto g_c( \nabla_{\pl_s}^i h,\nabla_{\pl_s}^i k) |c'|
\end{align*}
is smooth as well. Here we used again the module properties of Sobolev spaces. 
It remains to show the smoothness of $c\mapsto \ell_c$.
Therefore we use the fact that
\begin{align*}
L^{1}(D,\mathbb R)&\to\mathbb R\\
f&\mapsto \int f \ud \theta
\end{align*}
is a bounded linear operator, hence it immediately follows that the length function $c\mapsto \ell_c=\int |c'| \ud\theta$ is smooth.
For $n\geq 2$ the metric $G$ is a strong Riemannian metric on $\I^n(D,N)$ since for each 
$c\in \I^n(D,\N)$ the inner product $G_c(h,k)$ describes the Hilbert space structure on $T_c\I^n(D,N)$
(This is best seen in a local chart, whose base is, by definition, around a smooth $c$, otherwise one has to deal with $\Ga_{H^n}(c^*TN)$ for $c$ a Sobolev $H^n$-immersion).
\end{proof}

\subsection{Local well-posedness of the geodesic equation}
The local well-posedness results as summarized in the following theorem are based on the seminal method of Ebin and Marsden~\cite{ebin1970groups}.
They are known in the case of closed curves, see 
\cite{bauer2020fractional, michor2007overview,bauer2018fractional}, but to the best of our knowledge they are new for the case of open curves. 
However, as local well-posedness is not the focus of the current article, we postpone the proof of this result to Appendix~\ref{sec:local_well_posed}.

\begin{thm}\label{local_wellposedness}
Let $D=[0,2\pi]$ or $D=S^1$. Let $G$ be a Sobolev metric of order $n\geq 1$ of the form \eqref{def:SobMetric} on $\I^q(D,\N)$, with either $q\geq 2n$ or $q=n\ge 2$.
We have:
\begin{enumerate}
\item   
The initial value problem for the geodesic equation has unique local solutions in the Banach manifold $\I^q(D,\N)$. 
The solutions depend smoothly on $t$ and on the initial conditions $c(0, \cdot)$ and $c_t(0,\cdot)$. Moreover, the Riemannian exponential mapping $\operatorname{exp}$ exists and is smooth on a neighborhood of the zero section of the tangent bundle, 
and $(\pi, \operatorname{exp})$ is a diffeomorphism from a (possibly smaller) neighborhood of the zero section of $T\I^q(D,\N)$ to a neighborhood of the diagonal in $\I^q(D,\N)\times\I^q(D,\N)$.
\item
The results of part 1 (local well-posedness of the geodesic equation and properties of the exponential map) continue to hold on $\I^q(D,\N)\cap C^\infty(D^o,\N)$, where $D^o$ is the interior of $D$.
\end{enumerate}
\end{thm}
Note that for $D=S^1$ we have $\operatorname{Imm}(S^1,\N)=\I^q(S^1,\N)\cap C^\infty(S^1,\N)$, i.e., the local well-posedness continues to hold in the smooth category.

\section{Estimates}\label{sec:estimates}

In this section we prove some interpolation inequalities for Sobolev sections of the tangent bundle, that will be needed for proving metric completeness of $(\I^n(D,\N), G)$ in various cases.
For vector-space-valued functions, these inequalities are rather simple adaptations of standard inequalities; 
this is the case when $\N=\RR^d$, as sections of $c^*T\RR^d$ can be regarded as vector-space-valued functions (see \cite[Lemmas~2.14--2.15]{bruveris2014geodesic}, \cite[Lemma~2.4]{bruveris2015completeness} for the case $D=S^1$).

For a general target manifold, two things change:
first, instead of working with a section $h\in H^k(D,c^*T\N)$ directly, we need to parallel transport $h$ to a single base point, that is, to work with
\[
H(\theta) = \Pi_{\theta}^{\theta_0} h(\theta) \in H^k(D,T_{c(\theta_0)}\N)\simeq H^k(D,\RR^{\dim\N}),
\] 
where $\theta_0\in D$ is a base point, and $\Pi_{\theta}^{\theta_0}$ is the parallel transport, in $\N$, from $T_{c(\theta)}\N$ to $T_{c(\theta_0)}\N$, along $c$.
The reason for using $H$ is that it is a vector-space-valued function, and so we can take regular derivatives of $H$ and use the fundamental theorem of calculus.
The derivatives of $H$ relate to covariant derivatives of $h$ via
\begin{equation}\label{eq:nabla_Pi}
H'(\theta) = \frac{d}{\ud\theta} \Pi_\theta^{\theta_0} h(\theta) = \Pi_\theta^{\theta_0} \nabla_{\partial_\theta}h(\theta).
\end{equation}
See, e.g., \cite[Chapter~2, exercise 2]{carmo1992riemannian}.
Note that, since the parallel transport operator is an isometry, we have $|H(\theta)|=|h(\theta)|$, $|H'(\theta)| = |\nabla_{\pl_{\theta}} h(\theta)|$, and so on for higher order derivatives.

The second difference from the Euclidean case arises when $D=S^1$.
In the Euclidean case we obtain inequalities for periodic functions, that are generally better than the ones for general functions (and this fact is essential for completeness of constant coefficients metrics).
However, when $\N \ne \RR^d$, even though $h(0)=h(2\pi)$, it is not true that $H(0)=H(2\pi)$, because the holonomy along the curve $c$ is in general non-trivial (that is, $\Pi_{2\pi}^0 \ne \id_{T_{c(0)}\N}$).
Therefore, we need to bound the amount by which $H$ fails to be periodic, and to prove estimates for such ``almost periodic" functions.

We now state the estimates; first the inequalities that hold for both $D=S^1$ or $D=[0,2\pi]$, and then inequalities that hold only in the periodic case.
As the proof of the periodic case is long and somewhat different from the rest of the analysis in this paper, we postpone it to Appendix~\ref{app:estimates}.
This is done solely for the sake of readability
--- these estimates are at the core of proving the metric completeness of $(\I^n(S^1;\N);G)$ for $G$ with constant coefficients, and are one of the main differences between the analysis of manifold-valued curves and of $\RR^d$-valued curves.

\begin{lem}[General estimates]\label{lem:high_order_ineq}
If $n \geq 2$, $c \in \I^n(D,\N)$ and $h \in H^n(D,c^*T\N)$, then for $0 \leq k < n$, there exists $C=C(k,n,\dim\N)>0$ such that
\begin{equation}\label{eq:higher_order_ineq}
a^{2k} \| \nabla_{\pl_s}^k h \|^2_{L^2(\ud s)} \leq C\brk{\| h\|^2_{L^2(\ud s)} + a^{2n} \| \nabla_{\pl_s}^n h \|^2_{L^2(\ud s)}}\,,
\end{equation}
and
\begin{equation}\label{eq:higher_order_ineq_infty}
a^{2k} \| \nabla_{\pl_s}^k h \|^2_{L^\infty} \leq C\brk{a^{-1}\| h\|^2_{L^2(\ud s)} + a^{2n-1} \| \nabla_{\pl_s}^n h \|^2_{L^2(\ud s)}}\,,
\end{equation}
for every $a\in (0,\ell_c]$.
The same holds when we replace $\nabla_{\pl_s}$ with $\nabla_{\pl_\theta}$ and $\ud s$ with $\ud\theta$, with $a\in (0,2\pi]$.
\end{lem}
\begin{proof}
Since all the norms involved (in the $\ud s$ case) are reparametrization-invariant, we can assume that $c$ is arc-length parametrized.
In this case, we have $\nabla_{\pl_s} = \nabla_{\pl_\theta}$, $\ud s=\ud\theta$, where $\theta\in [0,\ell_c]$ (and in the case $D=S^1$, we identify the points $\theta=0$ and $\theta= \ell_c$).
Define
\[
H:[0,\ell_c] \to T_{c(0)}\N\equiv \RR^{\dim \N} \qquad H(\theta) = \Pi_\theta^0 h(\theta).
\]
From \eqref{eq:nabla_Pi} we have
\[
\Abs{\nabla_{\pl_\theta}^{k} h(\theta)} = \Abs{\Pi_\theta^0 \nabla_{\pl_\theta}^k h(\theta)} = \Abs{\frac{d^k}{\ud\theta^k} H(\theta)}.
\]
In order to prove \eqref{eq:higher_order_ineq}, we therefore need to prove that
\[
a^{2k} \int_0^{\ell_c} |\pl_\theta^k H|^2\,\ud\theta \le C\brk{\int_0^{\ell_c} | H|^2\,\ud\theta + a^{2n} \int_0^{\ell_c} |\pl_\theta^n H|^2\,\ud\theta},
\]
for every $a\in (0,\ell_c]$, and similarly for \eqref{eq:higher_order_ineq_infty}.
Since $H$ is valued in $\RR^{\dim\N}$, this is a standard Sobolev inequality, see, e.g., \cite[Theorem~7.40]{leoni2017first}.

The $\ud\theta$ case is similar, but simpler (no need to reparametrize $c$ first).
\end{proof}


\begin{lem}[Estimates for $S^1$]\label{lem:high_order_ineq_per}
If $n \geq 2$, $c \in \I^n(S^1,\N)$ and $h \in H^n(S^1,c^*T\N)$, then for $0 < k < n$, there exists $C>0$, depending on $k,n,\dim\N$, the injectivity radius and the upper and lower bounds on the sectional curvature of $\N$, such that
\begin{equation}\label{eq:higher_order_ineq_per}
\| \nabla_{\pl_s}^k h \|^2_{L^2(\ud s)} \leq C\min\BRK{1,\ell_c^2}\brk{\| h\|^2_{L^2(\ud s)} + \| \nabla_{\pl_s}^n h \|^2_{L^2(\ud s)}}.
\end{equation}
\end{lem}

\begin{proof}
See Appendix~\ref{app:estimates}.
\end{proof}

\begin{rem}\label{rem:simpler_inequalities}
It is interesting to compare inequality \eqref{eq:higher_order_ineq_per} to the equivalent one in the Euclidean settings \cite[Lemma~2.14]{bruveris2014geodesic}, that is, when $\N = \RR^d$.
There we have
\[
\| \nabla_{\pl_s} h \|^2_{L^2(\ud s)} \leq \frac{\ell_c^2}{4} \| \nabla_{\pl_s}^2 h \|^2_{L^2(\ud s)},
\]
from which higher order inequalities readily follow.
The zeroth order term that appears in the right-hand side of \eqref{eq:higher_order_ineq_per} is a curvature term, and, as the proof in Appendix~\ref{app:estimates} shows, arise from the non-trivial holonomy along the closed curve $c$.
\end{rem}

\section{Metric and geodesic completeness}\label{geodesic_metric_completeness}
We now want to prove the main result of this article, i.e.,
extend the completeness results, obtained for planar curves, to the situation studied in this article.
The exact statement of the main results is now detailed in Theorems~\ref{thm:metric_completeness}--\ref{thm:metric_completeness2} below (the main result as presented in the introduction is a slightly simplified form of them).
\begin{thm}\label{thm:metric_completeness}
Let $n\geq 2$, let $D=[0,2\pi]$ or $D=S^1$, and let $G$ be a smooth Riemannian metric on $\I^n(D,\N)$.
Assume that for every metric ball $B(c_0,r)\in (\I^n(D,\N),\operatorname{dist}^{G})$, there exists a constant $C=C(c_0,r)>0$, such that
for any $c\in B(c_0,r)$ and $h\in T_c \I^n(D,\N)$ we have
\begin{align}
\|h\|_{G_c}&\ge C\ell_c^{-1/2}\|\nabla_{\pl_s} h\|_{L^2(\ud s)},		\label{eq:G_vs_length_weighted_H_1}\\
\|h\|_{G_c}&\ge C\|\nabla_{\pl_s}^k h\|_{L^\infty} & k=0,\ldots,n-1, 	\label{eq:G_vs_L_infty_der} \\
\|h\|_{G_c}&\ge C\|\nabla_{\pl_s}^n h\|_{L^2(\ud s)}.			\label{eq:G_vs_H_n}
\end{align}
Then $G$ is a strong metric, and we have:
\begin{enumerate}
\item $(\I^n(D,\N),\dist^G)$ is a complete metric space.
\item $(\I^n(D,\N),G)$ is geodesically complete
\end{enumerate}
\end{thm}

For Sobolev metrics of the type~\eqref{def:SobMetric} we also obtain geodesic convexity:
\begin{thm}\label{thm:geodesic_existence}
Let $D=[0,2\pi]$ or $D=S^1$, and let $G$ be a smooth Sobolev metric of the type~\eqref{def:SobMetric} on $\I^n(D,\N)$, that satisfies assumptions \eqref{eq:G_vs_length_weighted_H_1}--\eqref{eq:G_vs_H_n}.
Then any two immersions in the same connected component can be joined by a minimizing geodesic.
\end{thm}
The reason that in Theorem~\ref{thm:geodesic_existence} we further assume, unlike in Theorem~\ref{thm:metric_completeness}, that $G$ is of the type~\eqref{def:SobMetric} is merely a technical one;
both theorems are first proved for metrics of this type, and in Theorem~\ref{thm:metric_completeness} the extension to the general case is immediate.
Theorem~\ref{thm:geodesic_existence}, with the same method of proof, definitely holds for metrics that are not of type~\eqref{def:SobMetric}, but this needs to be checked on a case-by-case basis, and thus we present this theorem only for this type of metrics. 
The assumptions~\eqref{eq:G_vs_length_weighted_H_1}--\eqref{eq:G_vs_H_n} are satisfied in the following cases:
\begin{thm}\label{thm:metric_completeness2}
Let $D=[0,2\pi]$ or $D=S^1$, and let $G$ be a Sobolev metric of order $n\geq 2$ of the type \eqref{def:SobMetric} on $\I^n(D,\N)$.
Assume that one of the following holds:
\begin{enumerate}
\item \textbf{Length weighted case:} There exists $\alpha>0$ such that either $a_1(x) \ge \alpha x^{-1}$ or both $a_0(x) \ge \alpha x^{-3}$ and $a_k(x) \ge \alpha x^{2k-3}$ for some $k>1$.\label{eq:length-weighted-conditons}
\item \textbf{Constant coefficient case:} $D=S^1$ and both $a_0$ and $a_n$ are positive constants.
\end{enumerate}
Then assumptions~\eqref{eq:G_vs_length_weighted_H_1}--\eqref{eq:G_vs_H_n} hold, and the completeness results of Theorem~\ref{thm:metric_completeness} hold for $(\I^n(D,\N),G)$.
\end{thm}

\begin{rem}
Note that the family of scale-invariant Sobolev metric, as introduced in Section~\ref{sec:main_result}, satisfies  conditions \eqref{eq:length-weighted-conditons} of Theorem~\ref{thm:metric_completeness2}. In the article~\cite{bruveris2017completeness}, where the authors study completeness properties for length weighted metrics on curves with values in Euclidean space, more general conditions on the coefficient functions that still ensure completeness have been derived. 
While such an analysis should be also possible in our situation, the resulting conditions would be much more complicated. 
The reason for this essentially lies in the fact that the manifold valued Sobolev estimates are more complicated (and involve lower-order terms), compared to the $\RR^d$-valued one, as described in Remark~\ref{rem:simpler_inequalities}.
Thus, for the sake of clarity, we discuss here only conditions of the type \eqref{eq:length-weighted-conditons}.
\end{rem}

The remaining part of this section will contain the proof of these theorems. 
To prove Theorem~\ref{thm:metric_completeness} we will first 
show the metric completeness, which then implies the geodesic completeness, see~\cite[VIII, Proposition~6.5]{lang2012fundamentals}. Since the theorem of Hopf-Rinow is not valid in infinite dimensions\footnote{Atkin constructed in \cite{atkin1975hopf} an example of a geodesically complete Riemannian manifold where the exponential map is not surjective, see also \cite{ekeland1978hopf}.} 
we cannot conclude the existence of geodesics by abstract arguments. 
Instead we show this statement by hand using the direct methods of the calculus of variations, in Section~\ref{sec:boundary_value_pblm}.
Finally, in Section~\ref{sec:smooth}, we deduce geodesic completeness in the smooth category.

\subsection{Reduction from metric completeness to equivalence of strong Riemannian metrics}\label{sec:reductions}
In this section we reduce the question of metric completeness $(\I^n(D,\N),\dist^G)$ to a question on uniform equivalence of the Riemannian metrics $G$ and $\calH$ on metric balls.
This is done in two steps.

\textbf{First reduction --- distance equivalence on balls.}
The space $\I^n(D,\N)$ is an open subset of $H^n(D,\N)$ (see Proposition~\ref{prop:Hn_Hilbert_manifold}).
In addition to the metric $\dist^{G}$ induced by $G$, it therefore inherits also the distance function $\dist^{\calH}$ induced from $H^n(D,\N)$.
In general, $\dist^G$ and $\dist^\calH$ are not equivalent.\footnote{This follows by the fact that $\calH$ and $G$ are no equivalent: For $G$ of the type \eqref{def:SobMetric}, the highest order derivative it involves is $\nabla_{\pl_s}^n$, which equals to $|c'|^{-n}\nabla_{\pl_\theta}^n$ plus lower order terms. The highest order derivative in $\calH$ is, on the other hand, $\nabla_{\pl_\theta}^n$. In particular, if we take a curve $c$ on which $|c'|$ is very close to being zero in some interval, it follows that $\calH_c$ and $G_c$ can have extremely large ratio.}
However, we do have the following:

\begin{prop}\label{prop:dist_equiv}
Assume that $G$ is a strong Riemannian metric on $\I^n(D,\N)$ and that the following holds:
\begin{enumerate}
\item For every metric ball $B(c_0,r)\subset \brk{\I^n(D,\N), \dist^G}$, there exists a constant $C>0$ such that $\dist^{\calH}\le C \dist^{G}$ on $B(c_0,r)$.
\item For every metric ball $B(c_0,r)\subset \brk{\I^n(D,\N), \dist^{G}}$, $\|{|c'|}^{-1}\|_{L^\infty}$ is bounded.
\end{enumerate}
Then $(\I^n(D,\N),\dist^{G})$ is metrically complete.
\end{prop}
\begin{proof}
The proof below is similar to the proof of \cite[Theorem~4.3]{bruveris2015completeness}. 
For the convenience of the reader we repeat the arguments here.
Given a Cauchy sequence $(c_n)$ in $\brk{\I^n(D,\N), \dist^{G}}$,  the sequence remains in a bounded metric ball in $\brk{\I^n(D,\N), \dist^{G}}$, hence by (1) the sequence is also a Cauchy sequence in $H^n(D,\N)$, hence $c_n\to c \in H^n(D,\N)$ (modulo a subsequence).
Moreover, since the sequence $c_n$ lies in a metric ball, $|c_n'|^{-1}<C<\infty$ for all $n$ by (2), and since $H^n$-convergence implies $C^1$-convergence, we obtain that $|{c'}|^{-1}\le C$, and thus $c\in \I^n(D,\N)$.
Since both $\calH$ and $G$ are strong metrics on $\I^n(D,\N)$, they induce the same topology (the manifold topology) \cite[VII, Proposition~6.1]{lang2012fundamentals}, and thus $\dist^\calH(c_n,c)\to 0$ implies that $\dist^G(c_n,c)\to 0$, hence $\brk{\I^n(D,\N), \dist^{G}}$ is metrically complete.
\end{proof}

\textbf{Second reduction --- metric equivalence implies distance equivalence.}

Next, we show that distance-equivalence on metric balls follows from metric-equivalence on metric balls. The following proposition is the content of Proposition~3.5 and Remark~3.6 in \cite{bruveris2015completeness}, adapted to our setting.
\begin{prop}\label{prop:metric_equiv}
Assume that for each metric ball 
$$B(c_0,r)\subset \brk{\I^n(D,\N), \dist^{G}},$$ 
there exists $C=C(c_0,r)>0$ such that for every $c\in B(c_0,r)$ and $h\in H^n(D;c^*T\N)$, we have
\begin{equation}\label{eq:metric_equiv}
\|h\|_{\calH_c} \le C\|h\|_{G_c}.
\end{equation}
Then, property (1) in Proposition~\ref{prop:dist_equiv} holds.
\end{prop}

\begin{proof}
The following proof is an adaptation of the proof of \cite[Lemma~4.2]{bruveris2015completeness}.
Let $c_1,c_2\in B(c_0,r)$ and $\e>0$, and let $\gamma$ be a piecewise smooth curve between $c_1$ and $c_2$ with $L^G(\gamma) \le \dist^G(c_1,c_2) +\e$.
Since $\dist^G(c_1,c_2)<2r$, by the triangle inequality, we have that $\gamma\subset B(c_0,3r)$.
We then have, using assumption~\eqref{eq:metric_equiv} for $B(c_0,3r)$, that
\[
\dist^\calH(c_1,c_2) \le L^\calH(\gamma) \le C L^G(\gamma) \le C(\dist^G(c_1,c_2) + \e).
\]
Since $\e$ is arbitrary, completes the proof.
\end{proof}

\subsection{Estimates on $\ell_c$ and $|c'|$ in metric balls}\label{sec:estimates_length_speed}

In this section we bound various quantities that depend on the curve $c$ uniformly on metric balls in $\I^n(D,\N)$.
These will enable us to prove the assumption of Proposition~\ref{prop:metric_equiv}, as well as assumption (2) of Proposition~\ref{prop:dist_equiv}.

To this end, we will repeatedly use the following result (see \cite[Lemma~3.2]{bruveris2015completeness} for a proof\footnote{In fact, for the case in which $C$ is independent of the metric ball, the statement in \cite[Lemma~3.2]{bruveris2015completeness} is inaccurate; the statement of Lemma~\ref{lem:local_bounded} is the corrected one, and the proof follows exactly as in \cite[Lemma~3.2]{bruveris2015completeness}, by checking carefully which constants appear when using Gronwall's inequality \cite[Corollary~2.7]{bruveris2015completeness}.}):
\begin{lem}\label{lem:local_bounded}
Let $(\mathcal{M},\mathfrak{g})$ be a Riemannian manifold, possibly of infinite dimension, and let $F$ be a normed space.
Let $f:\mathcal{M}\to F$ be a $C^1$-function, such that for each metric ball $B(y,r)$ in $\mathcal{M}$ there exists a constant $C$, such that
\[
\|T_xf.v\|_F \le C(1+\|f(x)\|_F)\|v\|_x \qquad \text{for all } \,\,x\in B(y,r), \,\, v\in T_x\mathcal{M}.
\]
Then $f$ is Lipschitz continuous on every metric ball, and in particular bounded on every metric ball.
Moreover, if the constant $C$ is independent of the metric ball $B(y,r)$, then the Lipschitz constant in $B(y,r)$ can be bounded by a function $L:[0,\infty)^3 \to (0,\infty)$, increasing in all variables, as follows:
\[
\| f(x_1) - f(x_2)\|_F \le L(C,\|f(y)\|_F,r) \dist(x_1,x_2)
\quad \text{ for every } x_1,x_2 \in B(y,r).
\]
In particular the Lipschitz constant in $B(y,r)$ depends on $y$ only through $\|f(y)\|_F$.
\end{lem}

\begin{rem}
Tracking the constants in Lemma~\ref{lem:local_bounded} carefully, one can obtain the bound
\begin{equation}\label{eq:Lip_const}
L(C,t,r) = C^2(1+r)(1+2r)e^{2Cr}(1+t).
\end{equation}
Note that this is not sharp, it is simply what is obtained by the method of the proof (using Gronwall's inequality).
\end{rem}

\begin{lem}[Bounds on length]
\label{lem:bounds_length}
Assume that assumption \eqref{eq:G_vs_length_weighted_H_1} holds. 
Then the length function $c\mapsto \ell_c$ is bounded from above and away from zero  on every metric ball.
\end{lem}

\begin{proof}
From \eqref{eq:ell_derivative} we have that
\begin{multline*}
|D_{c,h}\ell_c| \le \int_D |g(v,\nabla_{\pl_s} h)|\,\ud s \le \ell_c^{1/2} \brk{\int_D |g(v,\nabla_{\pl_s} h)|^2\,\ud s}^{1/2} \le
\\
\le \ell_c^{1/2} \|\nabla_{\pl_s} h\|_{L^2(\ud s)}.
\end{multline*}
Therefore, under the assumption \eqref{eq:G_vs_length_weighted_H_1},
we have
\[
|D_{c,h}\ell_c| \lesssim  \|h\|_{G_c}
\]
and we obtain from Lemma~\ref{lem:local_bounded} that $c\mapsto \ell_c$ is bounded on every metric ball.

Similarly, for the map $c\mapsto \ell_c^{-1}$, we have, under the assumption \eqref{eq:G_vs_length_weighted_H_1}, that
\[
|D_{c,h}\ell_c^{-1}| = \ell_c^{-2} |D_{c,h}\ell_c| \le \ell_c^{-3/2} \|\nabla_{\pl_s} h\|_{L^2(\ud s)} \lesssim \ell_c^{-1}\|h\|_{G_c},
\]
which concludes the proof using again Lemma~\ref{lem:local_bounded}.
\end{proof}

\begin{lem}[Bounds on speed]
\label{lem:bound_speed}
Assume that assumption \eqref{eq:G_vs_L_infty_der} holds for $k=1$.
Then, there exists a constant $\alpha=\alpha(c_0,r)>0$ such that
\[
\alpha^{-1} \le |c'(\theta)| \le \alpha
\]
for every $c\in B(c_0,r)$ and $\theta\in D$.
\end{lem}
\begin{proof}
Consider the function
\[
\log |c'| \,:\, (\I^n(D,\N), G) \to L^\infty(D;\RR).
\]
By \eqref{eq:speed_derivative} and assumption \eqref{eq:G_vs_L_infty_der} we have
\[
\|D_{c,h}\log |c'|\|_{L^\infty} \le \| g(v,\nabla_{\pl_s} h)\|_{L^\infty} \le \| \nabla_{\pl_s} h\|_{L^\infty} \lesssim \|h\|_{G_c}.
\]
By Lemma~\ref{lem:local_bounded} we thus have that $\log |c'|$ is bounded on metric balls, from which the claim follows.
\end{proof}

\begin{lem}\label{lem:bounded_image}
Assume that assumption \eqref{eq:G_vs_L_infty_der} holds for $k=0$.
Then the image in $\N$ of every metric ball $B(c_0,r)$ is bounded.
That is, there exists $R=R(c_0,r)>0$ such that for every $c\in B(c_0,r)$ and every $\theta\in D$,
\[
\dist_\N(c(\theta),c_0(0)) < R.
\]
\end{lem}

\begin{proof}
Let $c\in B(c_0,r)$, and let $c(t,\theta):[0,1] \to B(c_0,r)$ be a path from $c_0=c(0,\cdot)$ to $c=c(1,\cdot)$, whose length is smaller than $r$.
Using \eqref{eq:G_vs_L_infty_der}, we have
\[
\dist_\N(c(\theta),c_0(\theta)) \le \int_0^1 |\pl_t c(t,\theta)|\, \ud t \le C\int_0^1 \|\pl_t c(t,\theta)\|_{G_c} \,\ud t < Cr.
\]
This completes the proof, as the length of $c_0$ is finite.
\end{proof}

\begin{lem}\label{lem:high_order_bounds}
Assume that assumptions \eqref{eq:G_vs_length_weighted_H_1}--\eqref{eq:G_vs_H_n} hold.
Then the following quantities are bounded on every metric ball
\begin{eqnarray}
\|\nabla_{\pl_s}^k |c'|\|_{L^\infty} 			&\quad& k=0,\ldots, n-2 , \label{eq:bound_speed_L_infty} \\ 
\|\nabla_{\pl_s}^k |c'|\|_{L^2}			 	&\quad& k=0,\ldots, n-1 , \label{eq:bound_speed_L_2}
\end{eqnarray}
where $L^2$ is with respect to either $\ud s$ or $\ud\theta$.
\end{lem}
\begin{proof}
The proof of this is result follows by an induction on $k$ using iteratively Lemma~\ref{lem:bounds_length} and 
\ref{lem:bound_speed}. It is mainly an adaptation of Lemma~3.3 and Proposition~3.4 in \cite{bruveris2015completeness},
though the calculations in our situation are more involved due to the appearance of curvature terms of the manifold $\N$. To keep the presentation simple we postpone it to the Appendix~\ref{app:em:high_order_bounds}.
\end{proof}

\subsection{Proof of Theorem~\ref{thm:metric_completeness}: metric and geodesic completeness}\label{sec:pf_metric_comp}
We are now able to prove Theorem~\ref{thm:metric_completeness}, that is, that $(\I^n(D,\N),G)$ is metrically and geodesically complete.
We first prove it for a metric $G$ of the type \eqref{def:SobMetric} of order $n$ that satisfies assumptions \eqref{eq:G_vs_length_weighted_H_1}--\eqref{eq:G_vs_H_n}.
Afterwards the assumption that $G$ is of the type \eqref{def:SobMetric} will be removed. 

In particular, $G$ satisfies \eqref{eq:G_vs_L_infty_der}, and therefore Lemma~\ref{lem:bound_speed} implies that assumption (2) in Proposition~\ref{prop:dist_equiv} holds.
Therefore, in order to prove that $(\I^n(D,\N),\dist^{G})$ is metrically complete, we need to show that $G$ is a strong metric, and prove property (1), which by Proposition~\ref{prop:metric_equiv} follows from \eqref{eq:metric_equiv}. 
In fact, we will show a stronger result and prove that $G$ and $\calH$ are equivalent uniformly on metric balls. 
This will also imply that $G$ is a strong metric.

From Lemma~\ref{lem:high_order_ineq}, we have
\[
\sum_{i=0}^n \|\nabla_{\pl_\theta}^i h\|_{L^2(\ud\theta)}^2 \le C\|h\|_{\calH}^2
\]
for some universal constant $C>0$. Similarly, we have
\begin{equation}\label{eq:G_flat_vs_L_infty}
\|\nabla_{\pl_\theta} h \|_{L^\infty} \le C'\|h\|_{\calH}^2
\end{equation}
for some universal constant $C'>0$.

From the definition of $\nabla_{\pl_s}$ we have, by using the Leibniz rule,
\[
\nabla_{\pl_s}^k h = \frac{1}{|c'|^k}\nabla_{\pl_\theta}^k h + \sum_{i=1}^{k-1}P_{i,k}\nabla_{\pl_\theta}^i h,
\]
where $P_{i,k}$ are polynomials in $|c'|,\nabla_{\pl_s}|c'|,\ldots,\nabla_{\pl_s}^{k-i} |c'|$ and $|c'|^{-1},\ldots,|c'|^{-k}$, which are linear in $\nabla_{\pl_s}^{k-i}|c'|$.
Similarly,
\[
\nabla_{\pl_{\theta}}^k h = |c'|^k \nabla_{\pl_s}^k h + \sum_{i=1}^{k-1}Q_{i,k} \nabla_{\pl_s}^i h
\]
where $Q_{i,k}$ is a polynomial in the variables $|c'|,\nabla_{\pl_s}|c'|,\ldots,\nabla_{\pl_s}^{k-i} |c'|$ and the variables $|c'|,\ldots,|c'|^{k-1}$, which are linear in $\nabla_{\pl_s}^{k-i}|c'|$.
Using Lemma~\ref{lem:bound_speed} and Lemma~\ref{lem:high_order_bounds}, we therefore have that for $k<n$, $P_{i,k}$ and $Q_{i,k}$ are uniformly bounded on any metric ball, and so are $|c'|^{\pm 1}$, hence
\[
|\nabla_{\pl_s}^k h| \lesssim \sum_{i=1}^k |\nabla_{\pl\theta}^i h|, \quad 
|\nabla_{\pl\theta}^k h| \lesssim \sum_{i=1}^k |\nabla_k^i h|,
\]
uniformly on every metric ball.
The bound on $|c'|^{\pm 1}$ also implies that integration with respect to $\ud s$ or $\ud \theta$ are equivalent, hence
\begin{equation}\label{eq:nabla_s_k_vs_nabla_theta_k}
\|\nabla_{\pl_s}^k h\|_{L^2(\ud s)} \lesssim \sum_{i=1}^k \|\nabla_{\pl\theta}^i h\|_{L^2(\ud\theta)}, \quad 
\|\nabla_{\pl\theta}^k h\|_{L^2(\ud\theta)} \lesssim \sum_{i=1}^k \|\nabla_k^i h\|_{L^2(\ud s)},
\end{equation}
uniformly on every metric ball.

For $k=n$, we have, uniformly on every metric ball,
\[
|\nabla_{\pl_s}^n h| \lesssim \Abs{\nabla_{\pl_s}^{n-1} |c'|}|\nabla_{\pl_\theta}h| + \sum_{i=2}^n |\nabla_{\pl_\theta}^i h|
\]
and
\[
|\nabla_{\pl\theta}^n h| \lesssim \Abs{\nabla_{\pl_s}^{n-1} |c'|}|\nabla_{\pl_s} h| + \sum_{i=2}^n |\nabla_{\pl_s}^i h|,
\]
and therefore, invoking Lemma~\ref{lem:high_order_bounds} again and using \eqref{eq:G_flat_vs_L_infty}, we have,
\begin{equation}\label{eq:nabla_s_n_vs_nabla_theta_n}
\|\nabla_{\pl_s}^n h\|_{L^2(\ud s)} \lesssim \|\nabla_{\pl_\theta} h \|_{L^\infty} + \sum_{i=2}^n \|\nabla_{\pl_\theta}^i h\|_{L^2(\ud\theta)} \lesssim C\|h\|_{\calH}
\end{equation}
and, using \eqref{eq:G_vs_L_infty_der} again,
\begin{equation}\label{eq:nabla_s_n_vs_nabla_theta_n2}
\|\nabla_{\pl\theta}^n h\|_{L^2(\ud\theta)} 
\lesssim \|\nabla_{\pl_s}h \|_{L^\infty} + \sum_{i=2}^n \|\nabla_{\pl_s}^i h\|_{L^2(\ud s)}
\lesssim \|h\|_{G_c} + \sum_{i=2}^n \|\nabla_{\pl_s}^i h\|_{L^2(\ud s)}.
\end{equation}

Since \eqref{eq:G_vs_length_weighted_H_1} holds, we have by Lemma~\ref{lem:bounds_length} that $\ell_c$ is uniformly bounded from above and below on metric balls, hence all the coefficient functions $a_i(\ell_c)\ge 0$ are bounded from above on metric balls, and $a_0,a_n$ are also bounded away from zero.
We therefore have that, on each metric ball
\[
\|h\|_{L^2(\ud s)} + \|\nabla_{\pl_s}^n h\|_{L^2(\ud s)} \lesssim \|h\|_{G_c} \lesssim \sum_{i=0}^n \|\nabla_{\pl_s}^i h \|_{L^2(\ud s)}.
\]
Since $\ell_c$ is bounded from below and above uniformly on metric balls, Lemma~\ref{eq:higher_order_ineq} enables us to improve that to
\[
\sum_{i=0}^n \|\nabla_{\pl_s}^i h \|_{L^2(\ud s)} \lesssim \|h\|_{G_c} \lesssim \sum_{i=0}^n \|\nabla_{\pl_s}^i h \|_{L^2(\ud s)}
\]
Combining this with the estimate \eqref{eq:nabla_s_k_vs_nabla_theta_k}, \eqref{eq:nabla_s_n_vs_nabla_theta_n} and \eqref{eq:nabla_s_n_vs_nabla_theta_n2} immediately imply 
\[
\|h\|_{\calH_c} \lesssim \|h\|_{G} \lesssim \|h\|_{\calH_c},
\]
uniformly on metric balls.
In particular, this implies \eqref{eq:metric_equiv} and show that $G$ is a strong metric, thus all the assumptions of Propositions~\ref{prop:dist_equiv}--\ref{prop:metric_equiv} are satisfied,
which completes the proof of metric completeness.
As stated before, geodesic completeness follows directly as for strong Riemannian metrics (in infinite dimensions) metric completeness still implies geodesic completeness, see, e.g., \cite[VIII, Proposition~6.5]{lang2012fundamentals}. 

We now remove the assumption that $G$ is of the type \eqref{def:SobMetric}, and only assume that it is a smooth metric that satisfies \eqref{eq:G_vs_length_weighted_H_1}--\eqref{eq:G_vs_H_n}.
Denote by $\tilde{G}$ the metric
\[
\|h\|_{\tilde{G}_c}^2 := \|h\|_{L^2(\ud s)}^2 + \ell_c^{-1} \|\nabla_{\pl_s} h\|_{L^2(\ud s)}^2 + \|\nabla_{\pl_s}^n h\|_{L^2(\ud s)}.
\]
This metric is of the type \eqref{def:SobMetric}, and in Section~\ref{sec:pf_metric_comp2} below we show that this metric indeed satisfies \eqref{eq:G_vs_length_weighted_H_1}--\eqref{eq:G_vs_H_n}.
Therefore, it is metrically complete.

Now assume that $G$ is another metric that satisfies \eqref{eq:G_vs_length_weighted_H_1}--\eqref{eq:G_vs_H_n}.
We claim that on every metric ball $B^G(c_0,r)$, there exists a constant $C=C(c_0,r)$ such that $\|\cdot \|_{\tilde{G}_c} \le C \|\cdot\|_{G_c}$.
Indeed, assumptions \eqref{eq:G_vs_length_weighted_H_1} and \eqref{eq:G_vs_H_n} imply that $G$ controls the second and third addends in the definition on $\tilde{G}$;
since $\|h\|_{L^\infty} \ge \ell_c^{-1/2} \|h\|_{L^2(\ud s)}$, assumption 
\eqref{eq:G_vs_L_infty_der} for $k=0$ and Lemma~\ref{lem:bounds_length} imply that $G$ controls the second addend in $\tilde{G}$ as well (uniformly on every metric ball).
This implies, in particular, that $G$ is a strong metric (since $\tilde{G}$ is).

The proof is now concluded by similar arguments as Section~\ref{sec:reductions} (with $\tilde{G}$ instead of $\calH$):
Let $c_k\in (\I^n(D,\N), \dist^G)$ be a Cauchy sequence.
It follows that $c_k$ is also a Cauchy sequence in $(\I^n(D,\N), \dist^{\tilde{G}})$.
Since $(\I^n(D,\N), \dist^{\tilde{G}})$ is metrically complete, $c_k$ converges in $(\I^n(D,\N), \dist^{\tilde{G}})$ to some limit $c\in \I^n(D,\N)$.
Since both $G$ and $\tilde{G}$ are strong metrics on $\I^n(D,\N)$, they induce the same topology. 
Therefore, $c_k\to c$ in $(\I^n(D,\N), \dist^G)$ as well, thus proving metric completeness, from which geodesic completeness follows as before.

\subsection{Proof of Theorem~\ref{thm:metric_completeness2}}\label{sec:pf_metric_comp2}

\textbf{Length weighted case.}
If both $a_0(x) \ge \alpha x^{-3}$ and $a_k(x) \ge \alpha x^{2k-3}$ for some $k>1$, then by \eqref{eq:higher_order_ineq} we have that
\[
\ell_c^{-1}\|\nabla_{\pl_s} h \|_{L^2(\ud s)}^2 \le C\|h\|_G^2
\]
for some $C>0$. 
This is also obviously true if $a_1(x) \ge \alpha x^{-1}$.
Thus \eqref{eq:G_vs_length_weighted_H_1} holds, and from Lemma~\ref{lem:bounds_length} we obtain that the length function $c\mapsto \ell_c$ is bounded from above and away from zero on any metric ball.
Since $G$ is of the type \eqref{def:SobMetric}, we have that
\[
\|h\|_{G_c}^2 \ge a_0(\ell_c) \|h\|_{L^2(\ud s)}^2 + a_n(\ell_c) \|\nabla_{\pl_s}^n h\|_{L^2(\ud s)}^2,
\] 
and the bound on the length implies that on each metric ball, the constants $a_0(\ell_c)$ and $a_n(\ell_c)$ are bounded away from zero.
This immediately implies \eqref{eq:G_vs_H_n}, and also that on every metric ball
\[
\|h\|_{G_c}^2 \ge C( \|h\|_{L^2(\ud s)}^2 + \|\nabla_{\pl_s}^n h\|_{L^2(\ud s)}^2),
\] 
for some $C>0$.
On the other hand, using \eqref{eq:higher_order_ineq_infty} with $a=\ell_c$ we have, for every $k=0,\ldots, n-1$,
\[
\| \nabla_{\pl_s}^k h \|^2_{L^\infty} \leq C\brk{\ell_c^{-2k-1}\| h\|^2_{L^2(\ud s)} + \ell_c^{2(n-k)-1} \| \nabla_{\pl_s}^n h \|^2_{L^2(\ud s)}},
\]
hence on each metric ball, we have
\[
\| \nabla_{\pl_s}^k h \|^2_{L^\infty} \le C'\brk{\| h\|^2_{L^2(\ud s)} + \| \nabla_{\pl_s}^n h \|^2_{L^2(\ud s)}}\le C'' \|h\|_{G_c}^2,
\]
which implies \eqref{eq:G_vs_L_infty_der}.

\textbf{Constant coefficient case.}
Assume that $a_0$ and $a_n$ are positive constants.
We then immediately have \eqref{eq:G_vs_H_n}.
Furthermore, using \eqref{eq:higher_order_ineq_per} for $k=1$, we have
\[
\|\nabla_{\pl_s} h\|_{L^2(\ud s)}^2 \le C\ell_c^2 \|h\|_{G_c}^2
\]
for some constant $C$ that is independent of the curve $c$.\footnote{This is the crucial point in which the improved estimates for closed curves in Lemma~\ref{lem:high_order_ineq_per} are needed.}
This implies \eqref{eq:G_vs_length_weighted_H_1}, and hence the boundedness of $c\mapsto \ell_c$ by Lemma~\ref{lem:bounds_length}.
The proof of \eqref{eq:G_vs_L_infty_der} now follows in the same manner as the length weighted case.

\subsection{Proof of Theorem~\ref{thm:geodesic_existence}: existence of minimizing geodesics}\label{sec:boundary_value_pblm}
We now prove that any two immersions in the same connected component can be joined by a minimizing geodesic. 
The approach is a variational one: we consider the energy
\[
E(c) := \int_0^1 G_c(\dot c,\dot c) \,\ud t,
\]
defined on the set
\begin{multline*}
A_{c_0,c_1} := \Big\{ c:[0,1]\to \I^n(D,\N): \dot c\in L^2((0,1); H^n(D;c^*TN)),\\\qquad c(0) = c_0, c(1) = c_1\Big\},
\end{multline*}
where $c_0,c_1\in \I^n(D,\N)$ are two immersions in the same connected component (thus $A_{c_0,c_1}$ is a non-empty set).
We aim to show that there exists a minimizer to $E$ over $A_{c_0,c_1}$, which is, by definition, a minimizing geodesic.

We prove the existence of minimizers using the direct methods in the calculus of variations; namely, we take a minimizing sequence $c^j$, prove that it is weakly sequentially precompact, and that any limit point must be a minimizer.
In order to use weak convergence, we embed the curves in a Hilbert space, which neither $\I^n(D,\N)$ or $H^n(D,\N)$ are (this is the point where $\N$-valued curves differ from $\RR^d$-valued curves treated in \cite[Theorem~5.2]{bruveris2015completeness}). 
To this end, we again isometrically embed $\N$ into $\RR^m$ for some large enough $m\in \mathcal{N}$, as in the definition of $H^n(D,\N)$ that we started with (Definition~\ref{def:spaces}).
This will require us, as in Section~\ref{sec:spaces}, to use Lemma~\ref{lem:extrinsic_vs_intrinsic_norms} to relate the metric $\calH$ on $H^n(D,\N)$ with the standard Sobolev norm on $H^n(D;\RR^m)$.

Let now $c^j\in A_{c_0,c_1}$ be a minimizing sequence of $E$, that is,
\[
E(c^j) \to \inf_{A_{c_0,c_1}} E.
\]
In particular, $E(c^j)$ is a bounded sequence. Denote by $R^2$ its supremum.
We also fix an isometric embedding $\iota :\N\to \RR^m$, and, using this embedding, we consider $c^j$ as elements of the Hilbert space $H^1([0,1];H^n(D;\RR^m))$.

\noindent\textbf{Step I: The family $(c^j(t))_{j\in\mathbb{N}, t\in[0,1]}$ lies in a bounded ball around $c_0$.}
Fix $t_0\in [0,1]$ and $j\in \mathbb{N}$.
Since $c^j:[0,t_0] \to \I^n(D,\N)$ is a path from $c_0$ to $c^j(t_0)$, we have
\[
\dist_{G}^2(c^j(t_0),c_0) \le \brk{\int_0^{t_0} \|\dot c^j(t)\|_{G_{c^j(t)}} \ud t}^2 \le \int_0^1 \|\dot c^j(t)\|_{G_{c^j(t)}}^2\ud t = E(c^j) \le R^2.
\]
Therefore, $(c^j(t))_{j\in\mathbb{N}, t\in[0,1]} \subset B(c_0,R)$, where the ball is with respect to the metric $G$.

\noindent\textbf{Step II: The family $(c^j)_{j\in\mathbb{N}}$ is a bounded set in $H^1([0,1];H^n(D;\RR^m))$.}
Since $G$ satisfies \eqref{eq:G_vs_length_weighted_H_1}--\eqref{eq:G_vs_H_n}, we have that \eqref{eq:metric_equiv} hold uniformly on $B(c_0,R)$, that is, there exists $C>0$ such that
\[
C^{-1} \|h\|_{\calH_c} \le \|h\|_{G_c} \le C\|h\|_{\calH_c}, \quad \text{ for all } c\in B(c_0,R),\,\,h\in H^n(D;c^*T\N).
\]
This was proved in Section~\ref{sec:pf_metric_comp}.
Moreover, from Lemmata~\ref{lem:bounded_image}--\ref{lem:high_order_bounds}, we have that the assumptions of Lemma~\ref{lem:extrinsic_vs_intrinsic_norms} hold uniformly on $B(c_0,R)$, hence, combining with the above inequality, we obtain that there exists $C>0$ such that
\[
C^{-1} \|h\|_{H^n(\iota)} \le \|h\|_{G_c} \le C\|h\|_{H^n(\iota)}, \quad \text{ for all } c\in B(c_0,R),\,\,h\in H^n(D;c^*T\N).
\]
Since $(c^j(t))_{j\in\mathbb{N}, t\in[0,1]} \subset B(c_0,R)$, we obtain that for any fixed $t_0$ and $j$,
\[
\begin{split}
\|c_0 - c^j(t_0)\|_{H^n(\iota)}^2
	&= \int_D |c_0 - c^j(t_0)|^2 + |\pl_{\theta}^n(c_0-c^j(t_0))|^2 \,\ud \theta \\
	&= \int_D \Abs{\int_0^{t_0} \dot c^j(t)\,\ud t}^2 + \Abs{\int_0^{t_0} \pl_{\theta}^n \dot c^j(t_0)\,\ud t}^2 \,\ud \theta \\
	&\le \int_D\int_0^1 | \dot c^j(t)|^2 + | \pl_{\theta}^n \dot c^j(t_0)|^2 \,\ud t \,\ud \theta \\
	&= \int_0^1 \| \dot c^j(t)\|_{H^n(\iota)}^2 \,\ud t \\
	&= C\int_0^1 \| \dot c^j(t)\|_{G_{c^j(t)}}^2 \,\ud t \le CR^2.
\end{split}
\]
Therefore
\[
\begin{split}
\|c^j&\|_{H^1([0,1];H^n(D;\RR^m))}^2 
	= \int_0^1 \|c^j(t)\|_{H^n(\iota)}^2 + \|\dot c^j(t)\|_{H^n(\iota)}^2 \, \ud t \\ 
	&\le \int_0^1 2\|c_0\|_{H^n(\iota)}^2 + 2\|c_0 - c^j(t_0)\|_{H^n(\iota)}^2\, \ud t + \int_0^1 \|\dot c^j(t)\|_{H^n(\iota)}^2 \, \ud t \\
	&\le \int_0^1 2\|c_0\|_{H^n(\iota)}^2 + 2CR^2\, \ud t + C\int_0^1 \| \dot c^j(t)\|_{G_{c^j(t)}}^2 \, \ud t \\
	&\le 3CR^2 + 2\|c_0\|_{H^n(\iota)}^2
\end{split}
\]
Hence, the sequence $c^j$ is bounded in the Hilbert space $H^1([0,1];H^n(D;\RR^m))$.
Therefore, it has a subsequence (not relabeled) that weakly converges to some $c^*\in H^1([0,1];H^n(D;\RR^m))$.

\noindent\textbf{Step III: The limit point $c^*$ belongs to $A_{c_0,c_1}$.}
Let $\e\in (0,1/2)$.
We then have that the embedding $H^1([0,1];H^n(D;\RR^m))\subset C([0,1]; H^{n-\e}(D;\RR^m))$ is compact (due to the Aubin--Lions--Simon lemma\footnote{See, e.g., \cite[Theorem~II.5.16]{BF13}. With respect to the notation there we use the lemma for $p=\infty$, $r=2$, $B_0 = H^n$, $B_1=H^{n-\e}$ and $B_2=H^{n-1}$. We can use $p=\infty$ because $H^1$ embeds in $L^\infty$.}) and $H^{n-\e}(D;\RR^m)$ is compactly embedded in $C^{n-1}(D;\RR^m)$.
In particular, we thus have that $c^j\to c^*$ in the strong topology of $C([0,1];C^{n-1}(D;\RR^m))$.
Since $c^j(\theta)\in \N$ for all $j$ and $\theta$, the uniform convergence implies that $c^*(\theta)\in \N$ for all $\theta$ as well.
Since $c^j(0) = c_0$ and $c^j(1) = c_1$ for all $j$, the same holds for $c^*$.
Finally, since $c^j(t)\in B(c_0,R)$ for every $j$ and $t$, Lemma~\ref{lem:bound_speed} implies that
\[
|\pl_\theta c^j(t,\theta)| > \alpha
\]
for some $\alpha>0$.
Since $c^j\to c^*$ in $C([0,1];C^{n-1}(D;\RR^m))$, the same holds for $c^*$, hence $c^*\in \I^n(D,\N)$.
This shows that indeed $c^*\in A_{c_0,c_1}$.

\noindent\textbf{Step IV: Weak convergence of derivatives.}
It will be helpful now to emphasize the particular curve that is used to define the $\nabla_{\pl_s}$ derivative.
Therefore, for the rest of this proof, denote $D_{c^j}:=|c^j|^{-1}\nabla^\N_{\pl\theta}$.
We now show that, for $k=0,\ldots,n$, we have
\begin{equation}\label{eq:weak_conv_der}
D_{c^j}^n \dot c^j \rightharpoonup D_{c^*}^n \dot c^* \quad \text{ in }\, L^2([0,1];L^2(D;\RR^m)).
\end{equation}
By the definition of $c^*$, we have that
\[
\dot c^j \rightharpoonup \dot c^* \text{ in } L^2([0,1]; H^n(D;\RR^m)),
\]
hence the case $k=0$ is immediate.
We will show that for $k=1,\ldots,n$,
\begin{equation}\label{eq:weak_conv_der_aux0}
h^j \rightharpoonup h \text{ in } L^2([0,1]; H^k(D;\RR^m))
\end{equation}
implies
\begin{equation}\label{eq:weak_conv_der_aux}
D_{c^j} h^j \rightharpoonup D_{c^*} h \quad \text{ in }\, L^2([0,1];H^{k-1}(D;\RR^m)),
\end{equation}
from which \eqref{eq:weak_conv_der} follows by induction.
First, considering all the vector fields as sections of $D\times \RR^m$, we have that
\[
D_{c^j} h^j = \frac{1}{|\pl_\theta c^j|} \brk{\pl_\theta h^j - \II_{c^j}(\pl_\theta c^j, h^j)}, \quad
D_{c^*} h = \frac{1}{|\pl_\theta c^*|} \brk{\pl_\theta h - \II_{c^*}(\pl_\theta c^*, h)},
\]
where the subscript of $\II$ denotes the point where it is evaluated (recall that $\II$ is the second fundamental form of $\N$ in $\RR^m$).

Since $c^j\to c^*$ in $C([0,1];C^{n-1}(D;\RR^m))$ and $|\pl_\theta c^j|$ is uniformly bounded from below, we have that $|\pl_\theta c^j|^{-1}\to |\pl_\theta c^*|^{-1}$ uniformly (in $t$ and $\theta$).
In particular, since $\II_{c^j}$ are uniformly bounded bilinear forms (this follows again from Lemma~\ref{lem:bounded_image}),
it follows that $D_{c^j} h^j$ is a bounded sequence in $L^2([0,1];H^{k-1}(D;\RR^m))$.
Therefore, in order to prove \eqref{eq:weak_conv_der_aux}, it is enough to check it with respect to smooth test functions.
Let $u\in C([0,1]; C^\infty(D;\RR^m))$, and denote $w = u + (-1)^{k-1}\pl_\theta^{2k-2} u$;
we then have
\[
\inner{D_{c^j} h^j-D_{c^*} h, u}_{L^2([0,1];H^{k-1}(D;\RR^m))} 
	= \inner{D_{c^j} h^j-D_{c^*} h, w}_{L^2([0,1];L^2(D;\RR^m))}.
\]
Since $|\pl_\theta c^j|^{-1}\to |\pl_\theta c^*|^{-1}$ uniformly, the right-hand side converges to zero if
\[
\begin{split}
\inner{\pl_\theta h^j - \pl_\theta h, w}_{L^2([0,1];L^2(D;\RR^m))}&\to 0, \\
\inner{\II_{c^j}(\pl_\theta c^j, h^j) - \II_{c^*}(\pl_\theta c^*, h), w}_{L^2([0,1];L^2(D;\RR^m))}&\to 0.
\end{split}
\]
The first one follows from \eqref{eq:weak_conv_der_aux0}.
The second one follows also from \eqref{eq:weak_conv_der_aux0}, using in additon the fact that $c^j\to c^*$ in $C([0,1];C^{n-1}(D;\RR^m))$ implies that $\II_{c^j} \to \II_{c^*}$ uniformly, and $\pl_\theta c^j\to \pl_\theta c^*$ uniformly.
This completes the proof of \eqref{eq:weak_conv_der_aux}, and hence also of \eqref{eq:weak_conv_der}.

\noindent\textbf{Step V: $c^*$ is a minimizer.}
Using the embedding $\iota$, and considering all curves as curves in $\RR^m$, we can write the energy as
\[
\begin{split}
E(c) 	&=  \sum_{k=0}^n \int_0^1 \int_0^{2\pi} a_k(\ell_c) |D_c^k\dot c|^2 |\pl_\theta c| \,\ud \theta\, \ud t\\
	&=  \sum_{k=0}^n \|\sqrt{a_k(\ell_c)} \sqrt{|\pl_\theta c|} D_c^k\dot c\|_{L^2([0,1];L^2(D;\RR^m))}^2,
\end{split}
\]
where the transition to the second line uses the fact that $\iota$ is an isometric embedding.
Since $c^j\to c^*$ in $C([0,1];C^{n-1}(D;\RR^m))$, we have that $\sqrt{a_k(\ell_{c^j})}\to \sqrt{a_k(\ell_{c^*})}$ uniformly (for $k=0,\ldots,n$), and that $\sqrt{|\pl_\theta c^j|}\to \sqrt{|\pl_\theta c^*|}$ uniformly.
Therefore, \eqref{eq:weak_conv_der} implies that for all $k=0,\ldots, n$,
\[
\sqrt{a_k(\ell_{c^j})} \sqrt{|\pl_\theta c^j|} D_{c^j}^k\dot c^j \rightharpoonup \sqrt{a_k(\ell_{c^*})} \sqrt{|\pl_\theta c^*|} D_{c^*}^k\dot c^*
\quad \text{ in }\, L^2([0,1];L^2(D;\RR^m)).
\]
Since the map $x\mapsto \|x\|^2$ in a Hilbert space is weakly sequentially lower semicontinuous, we obtain that
\[
\inf_{A_{c_0,c_1}} E \le E(c^*) \le \liminf E(c^j) \to \inf_{A_{c_0,c_1}} E,
\]
hence $c^*$ is a minimizer.

\subsection{Geodesic completeness in the smooth category}\label{sec:smooth}

For closed curves, i.e., $D=S^1$ we obtain also completeness in the smooth category using the no-loss-no-gain result. 
\begin{cor}\label{cor:smooth}
Let $n\geq 2$ and let $G$ be a smooth Riemannian metric on $\I^n(S^1,\N)$. 
Assume that for every metric ball $B(c_0,r)\in (\I^n(S^1,\N),\operatorname{dist}^{G})$, there exists a constant $C=C(c_0,r)>0$ such that
conditions from theorem  \ref{thm:metric_completeness} hold, i.e., 
\begin{align*}
\|h\|_{G_c}&\ge C\ell_c^{-1/2}\|\nabla_{\pl_s} h\|_{L^2(\ud s)},
\tag{\ref{eq:G_vs_length_weighted_H_1}} \\
\|h\|_{G_c}&\ge C\|\nabla_{\pl_s}^k h\|_{L^\infty} & k=0,\ldots,n-1, 	\tag{\ref{eq:G_vs_L_infty_der}} \\
\|h\|_{G_c}&\ge C\|\nabla_{\pl_s}^n h\|_{L^2}.					\tag{\ref{eq:G_vs_H_n}}
\end{align*}
Then the space $(\Imm(S^1,\N),G|_{\Imm(S^1,\,\N)})$ is geodesically complete, where $G|_{\Imm(S^1,\,\N)}$ is the restriction of the metric $G$ to the space of smooth immersions.
\end{cor}
\begin{proof}
The proof of this result follows directly by applying Lemma~\ref{lem:no-loss-no-gain}, for $V=T\I^n(D,\N)$, an open subset of $H^n(D,T\N)$, and $F$ the exponential map of $G$.
\end{proof}

For open curves $D=[0,2\pi]$ one has to be slightly more careful, due to the potential loss of smoothness at the boundary; 
in this case Lemma~\ref{lem:no-loss-no-gain} only yields that solutions to the geodesic equation with smooth initial data remain at all times in
$\I^n([0,2\pi],\N)\cap C^\infty((0,2\pi),\N)$.

\section{Incompleteness of constant coefficient metrics on open curves}\label{sec:completion_open_curves}
In our main result we have seen a significant difference between open and closed curves: while we prove that the constant coefficient metrics of order $n\geq 2$ 
are geodesically and metrically complete on spaces of closed curves, for open curves we had to assume certain non-trivial length-weighted coefficients.
In fact,  for open curves with values in $\mathbb R^d$ 
it has been observed in~\cite[Remark~2.7]{bauer2019relaxed} that constant coefficient Sobolev metrics are in fact metrically incomplete, by 
constructing an explicit example of a path that leaves the space in finite time. 
Essentially, they showed that one can shrink a straight line to a point 
using finite energy. 
This behavior does not appear for closed curves as blow-up of curvature is an obstruction and thus ensures the completeness of the space.\footnote{This is true for metrics of order $n\ge 2$ that are discussed in this paper. Metrics of order $n<2$ are not strong enough to detect this curvature blowup, which results in metric- and geodesic-incompleteness, as seen in \cite[Section~6.1]{michor2007overview}.}
The goals of this section are twofold: 
\begin{enumerate}
\item to extend the example of metric incompleteness from \cite{bauer2019relaxed} (Example~\ref{ex:incompleteness}); 
\item to show that shrinking to a point is the only possibility to leave the space with finite energy (Theorem~\ref{thm:metric_completion}), and deduce from it a condition that ensures the existence of geodesics between given curves (Theorem~\ref{thm:bdry_prlm}). 
\end{enumerate}
The following example of metric incompleteness is a generalization of the example given in \cite[Remark~2.7]{bauer2019relaxed}. We only present  it for $\RR^2$-valued curves for the sake of clarity; it can be adapted easily to arbitrary target manifolds (disappearing along a geodesic instead of a straight line).
\begin{expl}\label{ex:incompleteness}
Consider $\I^n([0,2\pi];\RR^2)$ with the metric
\[
\|h\|_{G_c}^2 = \|h\|_{L^2(\ud s)}^2 + \|\nabla_{\pl_s}^n h \|_{L^2(\ud s)}^2.
\]
Consider the path $c:[0,1)\to \I^n$, defined by
\[
c(t,\theta) = ((1-t)(\theta-\pi) + f(t),g(t))
\]
for some smooth functions $f,g:[0,1)\to \RR$ to be determined.
Note that
\begin{align*}
c_\theta &= (1-t,0), &\qquad c_t &= (-(\theta-\pi) + f'(t),g'(t)), \\ 
\nabla_{\pl_s} c_t &= \brk{\frac{-1}{1-t},0}, & \nabla_{\pl_s}^k c_t &= 0 \text{ for } k>1.
\end{align*}
Hence
\begin{align*}
\|c_t\|_{G_c}^2 &= \int_{0}^{2\pi} \brk{(f'(t) - (\theta-\pi))^2 + g'(t)^2}(1-t)\,\ud\theta 
\\&
= 2\pi(1-t)\brk{\frac{\pi^2}{3} + f'(t)^2 + g'(t)^2},
\end{align*}
and therefore
\[
\begin{split}
\operatorname{length}(c) 
	&= \int_0^1 \|c_t\|_{G_c} 
		= \sqrt{2\pi}\int_0^1 (1-t)^{1/2}\brk{\frac{\pi^2}{3} + f'(t)^2 + g'(t)^2}^{1/2}\ud t \\
	&\le \sqrt{2\pi}\int_0^1 (1-t)^{1/2}\brk{\frac{\pi}{\sqrt{3}} + |f'(t)| + |g'(t)|}\ud t,
\end{split}
\]
hence $\operatorname{length}(c)<\infty$ if $\int_0^1 |f'(t)|(1-t)^{1/2} \,\ud t<\infty$ and similarly for $g$.
Under these restrictions on $f$ and $g$ many things can happen, for example:
\begin{enumerate}
\item For $f=g=0$ we obtain that $c$ converges, as $t\to 1$, to the constant curve at the origin;
\item For $f(t) = tx_0$ and $g(t) = ty_0$, $c$ converges to the constant curve at $(x_0,y_0)$.
\item For $f(t) = -\log(1-t)$ and $g=0$, $c$ converges to a point at infinity at the positive end of the $x$ axis.
\item For $f(t) = \sin(-\log(1-t))$ and $g=0$, $c$ does not converge pointwise to anything in $\RR^2$.
\end{enumerate}

Note that this analysis does not change if we replace $G$ with another constant coefficient metric.
This shows that $(\I^n([0,2\pi];\RR^2),\dist^G)$ is not metrically complete.
However, from the point of view of the metric completion, all these different choices of $f$ and $g$ are the same point in the completion --- indeed, let
\[
c^i(t,\theta) = ((1-t)(\theta-\pi) + f_i(t),g_i(t)), \quad i=1,2,
\]
and define, for a fixed $t\in [0,1)$, the path $\gamma^t(\tau,\theta)$ as the affine homotopy between $c^1(t,\cdot) = \gamma^t(0,\cdot)$ and $c^2(t,\cdot) = \gamma^t(1,\cdot)$, that is,
\[
\gamma^t(\tau,\theta) = ((1-t)(\theta-\pi) + \tau f_1(t) + (1-\tau) f_2(t), \tau g_1(t) + (1-\tau) g_2(t)).
\]
Since $|\gamma^t_\theta| = 1-t$ and $\gamma^t_\tau = (f_1(t) - f_2(t), g_1(t) - g_2(t))$ is independent of $\theta$ and $\tau$, it follows immediately that
\[
\operatorname{length}(\gamma^t) \propto 1-t.
\]
Therefore,
\[
\dist^G(c^1(t,\cdot),c^2(t,\cdot)) \le \operatorname{length}(\gamma^t) \propto 1-t \to 0
\]
as $t\to 1$.
This means, that in the metric completion, all the Cauchy sequences obtained by choosing different $f$s and $g$s are equivalent, hence converge to a single point.
\end{expl}

This example leads to the following open question:
\begin{qn}\label{qn:equiv}
Let $G$ be a constant coefficient Sobolev metric of order $n\geq 2$ of the type \eqref{def:SobMetric} on $\I^n([0,2\pi],\N)$.
For $i=1,2$, let $c^i_n \in \I^n([0,2\pi],\N)$ be two Cauchy sequences with $\ell_{c^i_n} \to 0$.
Does it hold that
\[
\lim_{n\to \infty} \dist^G(c^1_n,c^2_n) = 0 ?
\]
\end{qn}


We now show that if a Cauchy sequence of curves does not converge, its lengths must tend to zero: 
\begin{thm}\label{thm:metric_completion}
Let $G$ be a constant coefficient Sobolev metric of order $n\geq 2$ of the type \eqref{def:SobMetric} on $\I^n([0,2\pi];\N)$, where both $a_0$ and $a_n$ are strictly positive constants. 
Assume that $(c_n)_{n\in \mathbb{N}} \subset \I^n([0,2\pi];\N)$ is a Cauchy sequence with respect to $\dist^G$, whose lengths are bounded from below, that is $\ell_{c_n} > \delta>0$ for all $n$.
Then $c_n$ converges to some $c_\infty \in \I^n([0,2\pi];\N)$.
\end{thm}
Before proving this result we note a consequence of it:
if the answer to Question~\ref{qn:equiv} is positive, then, together with Theorem~\ref{thm:metric_completion}, it would give a positive answer to the following conjecture on the metric completion of $(\I^n([0,2\pi],\N),\dist^G)$:
\begin{qn}
Let $G$ be a constant coefficient Sobolev metric of order $n\geq 2$ of the type \eqref{def:SobMetric} on $\I^n([0,2\pi],\N)$. 
Is the metric completion of $(\I^n([0,2\pi],\N),\dist^G)$  given by $\I^n([0,2\pi],\N)\cup \{0\}$, where $\{0\}$ represents the limit of all vanishing-length Cauchy sequences?
\end{qn}

In our infinite dimensional situation metric incompleteness does not imply geodesic incompleteness. Furthermore the paths constructed in Example~\ref{ex:incompleteness}
are not geodesics (a direct calculations shows that the boundary equations in the geodesic equations are not satisfied). This leads to the following question:
\begin{qn}
Let $G$ be a constant coefficient Sobolev metric of order $n\geq 2$ of the type \eqref{def:SobMetric} on $\I^n([0,2\pi],\N)$.
Is $\I^n([0,2\pi],\N)$ geodesically complete?
\end{qn}

We proceed with the proof of Theorem~\ref{thm:metric_completion}. 
We will need the following lemma, which is similar to Lemma~\ref{lem:bounds_length}:
\begin{lem}\label{lem:bounds_length2}
Let $G$ be a Sobolev metric of order $n\geq 2$ on $\I^n([0,2\pi],\N)$, such that, for every $h\in H^n([0,2\pi],c^*T\N)$,
\[
\|\nabla_{\pl_s} h\|_{L^2(\ud s)} \le C\max\BRK{1,\ell_c^{-1}} \|h\|_{G_c}
\]
for some uniform constant $C>0$.
Then, the function $c\mapsto \ell_c^{3/2}$ is Lip\-schitz continuous on every metric ball in $(\I^n([0,2\pi],\N),\dist^{G})$.
Moreover, the Lipschitz constant of in $B(c_0,r)$ depends only on $\ell_{c_0}$ and $r$, and is an increasing function of both, that is, there exists a function $\textup{L}(C,\ell,r)$, increasing in all variables, such that
\[
|\ell_c^{3/2} - \ell_{\tilde{c}}^{3/2}| \le \textup{L}(C,\ell_{c_0},r) \dist^G(c,\tilde{c})
\quad \text{ for every } c,\tilde{c} \in B(c_0,r).
\]
\end{lem}
\begin{proof}[Proof of Lemma~\ref{lem:bounds_length2}]
As in the proof of Lemma~\ref{lem:bounds_length}, we have
\[
|D_{c,h} \ell_c^{3/2}| \le \frac{3}{2} \ell_c \|\nabla_{\pl_s} h\|_{L^2(\ud s)} \le \frac{3}{2}C\max\BRK{\ell_c,1} \|h\|_{G_c} \le \frac{3}{2}C(1+\ell_c^{3/2}) \|h\|_{G_c}
\]
from which the claim follows by Lemma~\ref{lem:local_bounded}, with $\textup{L}(C,\ell,r) := L(C,\ell^{3/2},r)$.
\end{proof}
\begin{proof}[Proof of Theorem~\ref{thm:metric_completion}]
Assume that $c_n$ is a Cauchy sequence with $\ell_{c_n} > \delta$ for some $\delta >0$.

Since $G$ has constant coefficients (with $a_0,a_n>0$), we have, using \eqref{eq:higher_order_ineq} for $k=1$, that
\[
\begin{split}
\|\nabla_{\pl_s} h\|_{L^2(\ud s)}^2 
	&\le C \max\BRK{1,\ell_c^{-2}}\brk{\|h\|_{L^2(\ud s)}^2 + \|\nabla_{\pl_s}^n h\|_{L^2(\ud s)}} \\
	&\le C' \max\BRK{1,\ell_c^{-2}} \|h\|_{G_c}^2,
\end{split}
\]
where the constants $C,C'$ depend only on $n$, $a_0$ and $a_n$, hence we can apply Lemma~\ref{lem:bounds_length2}.

There exists $N_1$ large enough such that $c_n \in B(c_{N_1},1/2)$ for all $n\ge N_1$.
Applying Lemma~\ref{lem:bounds_length2}, for $B(c_{N_1},1)$ we obtain that there exists a constant $\bar\ell$, depending on $c_{N_1}$ such that
\[
\ell_{c}\le \bar\ell \qquad \text{ for all } c\in B(c_{N_1},1).
\]
In particular, this applies to all $c_n$ for $n\ge N_1$.

Let $\textup{L}(C',\ell,r)$ be the Lipschitz constant bound as in Lemma~\ref{lem:bounds_length2}, and denote $\bar L := \textup{L}(C',\bar\ell,1)$.
Denote $r_0 := \min\BRK{\frac{\delta^{3/2}}{2\bar L},1/2}$.
There exists an index $N_2>N_1$ such that for $n\ge N_2$ we have that $c_n \in B(c_{N_2}, r_0/3)$, that is $\dist^G(c_n,c_N) < r_0/3$.
Applying Lemma~\ref{lem:bounds_length2} to $B(c_{N_2}, r_0)$ and the bound $\ell_{c_{N_2}} \le \bar\ell$, we have that
\[
\Abs{\ell_{c}^{3/2} - \ell_{\tilde{c}}^{3/2}} \le \bar L \dist^G(c,\tilde{c}) \text{ for every } c,\tilde{c}\in B(c_{N_2}, r_0).
\]
Since $\ell_{c_{N_2}} > \delta$, and $B(c_{N_2},r_0) \subset B(c_{N_1},1)$, we obtain
\begin{equation}\label{eq:length_bound_local}
\ell_{c} \in \Brk{\frac{\delta}{2^{2/3}}, \bar\ell} \qquad \text{ for every } c\in B(c_{N_2},r_0).
\end{equation}
Denote by $G'$ the standard scale-invariant metric of order $n$ on $\I^n([0,2\pi],\N)$; that is,
\[
\|h\|_{G'_c}^2 = \ell_c^{-3} \|h\|_{L^2(\ud s)}^2 + \ell_c^{2n-3} \|\nabla_{\pl_s}^n h \|_{L^2(\ud s)}^2.
\]
Recall that $(\I^n,\dist^{G'})$ is metrically complete by Theorem~\ref{thm:metric_completeness2}.
Using \eqref{eq:length_bound_local} and Lemma~\ref{lem:high_order_ineq}, it follows that $G'$ and $G$ are equivalent in $B(c_{N_2}, r_0)$.
From here we continue in a similar way as in Propositions~\ref{prop:dist_equiv}--\ref{prop:metric_equiv}:
Let $c,\tilde{c}\in B(c_{N_2},r_0/3)$, and $0<\e<r_0/3 - \dist^G(c,\tilde{c})$. 
Let $\gamma$ be a curve between $c$ and $\tilde{c}$ such that $\operatorname{length}^G(\gamma) < \dist^G(c,\tilde{c}) + \e$.
By triangle inequality, we have that $\gamma\subset B(c_{N_2},r_0)$, and since $G'$ and $G$ are equivalent there, we have that for some constant $C>0$ (independent of $\gamma$),
\[
\dist^{G'}(c,\tilde{c}) \le \operatorname{length}^{G'}(\gamma) \le C\operatorname{length}^G(\gamma) < C( \dist^G(c,\tilde{c}) + \e), 
\]
and since $\e$ is arbitrarily small, we conclude that
\[
\dist^{G'}(c,\tilde{c}) \le C\dist^G(c,\tilde{c}), \qquad \text{for every } c,\tilde{c} \in B(c_{N_2},r_0/3).
\]
Since for every $n\ge N_2$, $c_n\in B(c_{N_2},r_0/3)$, it follows that $c_n$ is a Cauchy sequence with respect to $G'$ as well.
Since $(\I^n,\dist^{G'})$ is metrically complete, we have that there exists $c_\infty \in \I^n$ such that $\dist^{G'}(c_n,c_\infty) \to 0$.
Since both $G$ and $G'$ are strong metrics on $\I^n$, they induce the same topology \cite[VII, Proposition~6.1]{lang2012fundamentals}, and thus $c_n \to c_\infty\in \I^n$ with respect to $G$ as well, which completes the proof.
\end{proof}

From the arguments in the proof of Theorem~\ref{thm:metric_completion}, we also obtain that for close enough immersions $c_0,c_1\in \I^n([0,2\pi];\N)$, there exists a connecting minimizing geodesic:
\begin{thm}\label{thm:bdry_prlm}
Let $G$ be a constant coefficient Sobolev metric of order $n\geq 2$ of the type \eqref{def:SobMetric} on $\I^n([0,2\pi];\N)$, where both $a_0$ and $a_n$ are strictly positive constants.
Let $c_0\in \I^n([0,2\pi];\N)$.
Then, there exists a constant $r_0$, depending only on the coefficients $a_k$ and on $\ell_{c_0}$, such that for every $c_1\in B(c_0,r_0)$, there exists a minimizing geodesic between $c_0$ and $c_1$.
\end{thm}

\begin{rem}
The proof below, together with the bound \eqref{eq:Lip_const}, imply that $r_0$ can be chosen such that
\[
r_0 = C \frac{\ell_{c_0}^{3/2}}{1+\ell_{c_0}^{3/2}} \ge C\min\brk{\ell_{c_0}^{3/2},\frac{1}{2}},
\]
where $C$ depends only on the coefficients $a_k$, $k=0,\ldots,n$.
Note that we do not know whether the existence of minimizing geodesics fails in general; it might be that although the space in metrically incomplete, a minimizing geodesic between any two curves $c_0,c_1\in \I^n([0,2\pi],\N)$ exists.
\end{rem}

\begin{proof}
As in Theorem~\ref{thm:metric_completion}, there exists a constant $C$, depending only on $n$, $a_0$ and $a_n$ (or alternatively, on $a_k$, $k=0,\ldots,n$) such that the assumption of Lemma~\ref{lem:bounds_length2} holds.
Fix $\tilde{L}:= \textup{L}(C,\ell_{c_0},1)$, where $\textup{L}(C,\ell,r)$ is the Lipschitz constant function from Lemma~\ref{lem:bounds_length2}.
Let
\[
r_0 = \min\brk{\frac{\ell_{c_0}^{3/2}}{2\tilde{L}},1}.
\]
It follows that
\[
\ell_{c} \in \Brk{\frac{1}{2^{2/3}}\ell_{c_0}, \frac{3^{2/3}}{2^{2/3}}\ell_{c_0}} \quad \text{ for every } c\in B(c_{0},r_0).
\]
As in Theorem~\ref{thm:metric_completion}, it follows that in this ball $G$ is uniformly equivalent to a scale-invariant Sobolev metric of order $n$ on $\I^n([0,2\pi];\N)$, hence Lemmata~\ref{lem:bound_speed}--\ref{lem:high_order_bounds} hold uniformly on $B(c_{0},r_0)$ (rather than on every metric ball).

Let $c_1\in B(c_0,r_0)$.
Define the energy $E(c)$ and the set of paths $A_{c_0,c_1}$ as in Section~\ref{sec:boundary_value_pblm}.
Let $c\in A_{c_0,c_1}$, with $\operatorname{length}(c) < r_0$.
Assume that $c$ has constant speed; we then have
\[
E(c) = \operatorname{length}(c)^2 < r_0^2.
\]
Therefore,
\[
\inf_{A_{c_0,c_1}} E < r_0^2.
\]
We can now take a minimizing sequence $c^j\in A_{c_0,c_1}$, and assume without loss of generality that $E(c^j) < r_0^2$ for all $j$.
The proof now follows in the same way as in Theorem~\ref{thm:metric_completion}.
\end{proof}


\appendix
\section{The geodesic equation}\label{app:geodesicequation}

\subsection{Proof of Lemma~\ref{thm:geodesicequation}: the geodesic equation}\label{app:proof_geodesicequation}
\begin{proof}[Proof of Lemma~\ref{thm:geodesicequation}]
To prove the formula for the geodesic equation we consider the energy of a path of immersions $c(t,\theta)$. Furthermore, we will treat the zeroth and first order terms separately.  Varying $c(t,\theta)$ in direction $h(t,\theta)$ with $h(0,\theta)=h(1,\theta)=0$
we obtain for the zeroth-order term:
\begin{align*}
&d\left(\int_0^1 a_0(\ell_c) \int_D g(c_t,c_t)|c'| \ud\theta \ud t\right)(h)\\
	&\qquad=\int_0^1 a_0'(\ell_c)D_{c,h}\ell_c\int_D  g(c_t,c_t)|c'| \ud\theta\ud t \\
	&\qquad\qquad\qquad+ \int_0^1 a_0(\ell_c) \int_D  2 g(\nabla_{h}c_t,c_t) + g(c_t,c_t)g(v,\nabla_{\pl_s}h) \ud s\ud t\\
&\qquad=
\int_0^1 \brk{a_0'(\ell_c)\int_D g(v, \nabla_{\pl_s} h) \ud s \int_D  g(c_t,c_t) \ud s} \ud t \\
&\qquad\qquad\qquad
+\int_0^1 \brk{a_0(\ell_c)\int_D  2 g(\nabla_{\partial_t}h,c_t) + g(\nabla_{\pl_s}h,vg(c_t,c_t))  \ud s}\ud t.\\
\end{align*}
where we used in the last step that
\begin{equation}
 \nabla_{h}c_t= \nabla_{\partial_t} h\;.
\end{equation}
and the variation formula for the length $\ell_c$ from Lemma~\ref{lem:variationformulas}.
Here, as before, $v=c'/|c'|$ is the unit length tangent vector to the curve $c$.

Similarly we calculate for the first-order terms:
\begin{align*}
&d\left(\int_0^1 a_1(\ell_c)\int_D   g(\nabla_{\pl_s} c_t,\nabla_{\pl_s} c_t)|c'| \ud\theta \ud t\right)(h)\\
&=\int_0^1 a_1'(\ell_c) \brk{\int_D g(v, \nabla_{\pl_s} h) \ud s \int_D g(\nabla_{\pl_s} c_t,\nabla_{\pl_s} c_t)\ud s} \ud t
\\&\;
+\int_0^1 a_1(\ell_c)\brk{ \int_D 2g(\nabla_{h}\nabla_{\pl_s} c_t,\nabla_{\pl_s} c_t)+ g(\nabla_{\pl_s} c_t,\nabla_{\pl_s} c_t) g(v,\nabla_{\pl_s}h) \ud s}\ud t
\\&
= \int_0^1  a_1'(\ell_c)\brk{\int_D g(v, \nabla_{\pl_s} h) \ud s \int_D g(\nabla_{\pl_s} c_t,\nabla_{\pl_s} c_t)\ud s} \ud t
\\&\;
+\int_0^1 a_1(\ell_c) \brk{\int_D 2 g(-g(v, \nabla_{\pl_s} h)  \nabla_{\pl_s} c_t + \nabla_{\pl_s} \nabla_h c_t+\mathcal R(v,h)c_t,\nabla_{\pl_s} c_t) \ud s} \ud t
\\&\;
+\int_0^1 a_1(\ell_c) \brk{\int_D  g(\nabla_{\pl_s} c_t,\nabla_{\pl_s} c_t) g(v,\nabla_{\pl_s}h) \ud s}\ud t
\\&
=\int_0^1 a_1'(\ell_c)\brk{\int_D g(v, \nabla_{\pl_s} h) \ud s \int_D g(\nabla_{\pl_s} c_t,\nabla_{\pl_s} c_t)\ud s} \ud t
\\&\;
+\int_0^1 a_1(\ell_c) \brk{ \int_D g(-g(v, \nabla_{\pl_s} h)  \nabla_{\pl_s} c_t + 2\nabla_{\pl_s} \nabla_{\partial_t}h+2\mathcal R(v,h)c_t,\nabla_{\pl_s} c_t) \ud s} \ud t.
\end{align*}
Sorting this by derivatives of $h$ we obtain
\begin{align*}
&d\left(\int_0^1  a_1(\ell_c)\int_D  g(\nabla_{\pl_s} c_t,\nabla_{\pl_s} c_t)|c'| \ud\theta \ud t\right)(h)\\
&=\int_0^1 a_1(\ell_c)\int_D  2g(\nabla_{\pl_s} \nabla_{\partial_t}h,\nabla_{\pl_s} c_t)+ 
 2g(\mathcal R(v,h)c_t,\nabla_{\pl_s} c_t)\\&\qquad +g\left(\nabla_{\pl_s} h,\frac{a_1'(\ell_c)}{a_1(\ell_c)}
 \int_D g(\nabla_{\pl_s} c_t,\nabla_{\pl_s} c_t)\ud s\; v -g(\nabla_{\pl_s} c_t,\nabla_{\pl_s} c_t)v
 \right) \ud s \ud t.
\end{align*}
Putting both together we obtain:
\begin{multline*}
dE(c).h=
\int_0^1\int_D  2a_0(\ell_c)g(\nabla_{\partial_t}h,c_t)+2a_1(\ell_c)g(\nabla_{\pl_s} \nabla_{\partial_t}h,\nabla_{\pl_s} c_t)\\+
 2a_1(\ell_c)g(\mathcal R(v,h)c_t,\nabla_{\pl_s} c_t) 
 +g(\nabla_{\pl_s} h,\Psi_c(c_t,c_t)v) \ud s\ud t,
 \end{multline*}
 where 
 \begin{align*}
 \Psi_c(c_t,c_t)
 	&=a_0(\ell_c)g(c_t, c_t)+ a_0'(\ell_c)\int_D  g(c_t,c_t) \ud s \\
 	&\quad-a_1(\ell_c)g(\nabla_{\pl_s} c_t,\nabla_{\pl_s} c_t)+a_1'(\ell_c) \int_D g(\nabla_{\pl_s} c_t,\nabla_{\pl_s} c_t)\ud s.
\end{align*}
To obtain the geodesic equation, we have to integrate by parts to free $h$ from all derivatives. We will treat the four terms separately. For the first two terms we 
recall that $\ud s$ depends on the curve $c$ (and thus on time $t$), i.e., for time dependent vector fields $h$ and $k$, with $h(0,\theta)=h(1,\theta)=0$, we have
\begin{equation}\label{int_by_parts_t}
\begin{split}
\int_0^1 \int_D g(\nabla_{\partial_t}h,k) \ud s\ud t&= -\int_0^1 \int_D g(h,\nabla_{\pl_t}(|c'|k)) \ud \theta\ud t\\&=
 -\int_0^1 \int_D g(h,\nabla_{\partial_t} k +g(v,\nabla_{\pl_s} c_t)k ) \ud s\ud t\;.
\end{split}
\end{equation}
Applying this formula to the first term yields:
\begin{multline*}
 2\int_0^1  \int_D g(\nabla_{\partial_t}h, a_0(\ell_c) c_t)\ud s\ud t
\\
 	 = -2\int_D a_0(\ell_c)\,g\Big(h,\nabla_t c_t+g(v,\nabla_{\pl_s} c_t) c_t 
\\
+\tfrac{a_0'(\ell_c)}{a_0(\ell_c)}\int_D g(v, \nabla_{\pl_s} c_t) \ud s\;  c_t\Big)\ud s\ud t\;.
\end{multline*}
For the second term we need to apply integration by parts in space first:
\begin{align*}
 &2\int_0^1\int_D  g(\nabla_{\pl_s} \nabla_{\partial_t}h,a_1(\ell_c) \nabla_{\pl_s} c_t) \ud s \ud t
 \\&= 2\int_0^1  g(\nabla_{\partial_t}h,a_1(\ell_c)\nabla_{\pl_s} c_t)|^{2\pi}_0 \ud t-
	2\int_0^1\int_D  g( \nabla_{\partial_t}h,a_1(\ell_c)\nabla_{\pl_s} \nabla_{\pl_s} c_t) \ud s \ud t
\\&=-2\int_0^1  g(h,\nabla_{\partial_t}\brk{a_1(\ell_c)\nabla_{\pl_s} c_t})|^{2\pi}_0 \ud t+
 		 2\int_0^1\int_D a_1(\ell_c) g(h,\nabla_{\partial_t}\nabla^2_s c_t) \ud s \ud t
\\&\;
+2\int_0^1\!\!\int_D \!a_1(\ell_c) g\!\left(h,g(v,\nabla_{\pl_s} c_t)\nabla^2_s c_t+\tfrac{a_1'(\ell_c)}{a_1(\ell_c)}\int_D g(v, \nabla_{\pl_s} c_t) \ud s\;\nabla^2_s c_t \right)\! \ud s \ud t
\end{align*}
For the third term we use the symmetries of the curvature tensor to obtain
\begin{align*}
2\!\!\int_0^1\!\! \int_D\!\! a_1(\ell_c) g(\mathcal R(v,h)c_t, \nabla_{\pl_s} c_t)\ud s\ud t
= 2\!\! \int_0^1 \!\!\int_D a_1(\ell_c) g(\mathcal R(c_t,\nabla_{\pl_s} c_t)v,h)\ud s\ud t.
\end{align*}
Finally for the last term we need to integrate in parts in space again, taking into account the boundary terms:
\begin{align*}
&\int_0^1\int_D g(\nabla_{\pl_s} h,\Psi_c(c_t,c_t)v) \ud s\ud t \\
	&\qquad = \int_0^1 g\big(h , \Psi_c(c_t,c_t)v)\Big|^{2\pi}_0\ud t
 			-\int_0^1\int_D g\big( h, \nabla_{\pl_s}(\Psi_c(c_t,c_t)v)\big)\ud s\ud t \,.
 \end{align*}
We can now read off the geodesic equation. We will fist start by collecting the terms on the interior of $D$:
\begin{align*}
&a_0(\ell_c)\nabla_{\partial_t} c_t-a_1(\ell_c)\nabla_{\partial_t}\nabla^2_s c_t\\&\qquad=-a_0(\ell_c)g(v,\nabla_{\pl_s} c_t) c_t-a_0'(\ell_c)\brk{\int_D g(v, \nabla_{\pl_s} c_t) \ud s}  c_t\\
&\qquad\qquad\qquad+ a_1(\ell_c)g(v,\nabla_{\pl_s} c_t) \nabla_{\pl_s}^2 c_t+ a_1'(\ell_c)\brk{\int_D g(v, \nabla_{\pl_s} c_t) \ud s}\nabla^2_s c_t
\\&\qquad\qquad\qquad+a_1(\ell_c)\mathcal R(c_t,\nabla_{\pl_s} c_t)v  -\frac12\nabla_{\pl_s}(\Psi_c(c_t,c_t)v).
\end{align*}
From here the result follows using the definition of the inertia operator $A_c$, the product rule for the term $\nabla_{\pl_s}(\Psi_c(c_t,c_t)v)$, 
by using the formula
\begin{align*}
\nabla_t (A_c c_t) &=(\nabla_t A_c)c_t+A_c(\nabla_t c_t)=   \nabla_t (a_0(\ell_c)c_t - a_1(\ell_c)\nabla_{\pl_s}^2c_t)\\&=\brk{\int_D g(v, \nabla_{\pl_s} c_t) \ud s}\; a_0'(\ell_c) c_t+
a_0(\ell_c)\nabla_{\partial_t} c_t\\&\qquad-\brk{\int_D g(v, \nabla_{\pl_s} c_t) \ud s}\; a_1'(\ell_c)\nabla^2_s c_t -a_1(\ell_c)\nabla_{\partial_t}\nabla^2_s c_t\;,
\end{align*}
and by collecting the boundary terms if $D=[0,2\pi]$.
\end{proof}

\subsection{Proof of Theorem~\ref{local_wellposedness}: local well-posedness}\label{sec:local_well_posed}
In this section we will use the method of Ebin--Marsden to obtain local well-posedness and uniqueness of the geodesic equations. 
Before we prove the local well-posedness we formulate a variant of the no-loss-no-gain result, which is also used in Section~\ref{sec:smooth}.
\begin{lem}\label{lem:no-loss-no-gain}
Let $q\geq 2$, $V \subset H^q(D,T\N)$ an open subset and let 
$F: V \to H^q(D,T\N)$ be a smooth and $\mathcal D^q(D)$ equivariant map, i.e., $F(h\circ \varphi)=F(h)\circ \varphi$ for all $h\in H^q(D,T\N)$ and $\varphi\in \mathcal D^q(D)$. Then $F$ is a smooth map from $V\cap H_{\operatorname{loc}}^{q+l}(D^o,T\N)$ to itself
for any $l\in \mathbb N$, where $D^o$ is the interior of $D$.
\end{lem}
\begin{proof}
For $D=S^1$ this result is shown in~\cite[Corollary~4.1]{bruveris2017regularity}. 
For the case $D=[0,2\pi]$ the proof is essentially the same, see also the arguments of Ebin and Marsden \cite[Theorem~12.1, Lemma~12.2]{ebin1970groups} who proved the original no-loss-no-gain results for manifolds with boundary.
\end{proof}

\begin{proof}[Proof of~Theorem~\ref{local_wellposedness}]
For closed curves, i.e., $D=S^1$, this result can be found in~\cite[Theorem 4.4]{bauer2020fractional}, see also~\cite{michor2007overview,bauer2018fractional}.
In the following we will focus on the case of open curves, where the proof will be slightly more involved due to the existence of a boundary. 
For a strong Riemannian metric  ($q=n$) the existence of the geodesic equation and its local well-posedness is always guaranteed, see, e.g., \cite[VIII, Theorem~4.2]{lang2012fundamentals}.
Thus we obtain the first part of the theorem for the Sobolev metric of order $n\geq2$ on $\I^n([0,2\pi],\N)$ by Theorem~\ref{thm:metric_sobcompletion}.
For $q\neq n$ we have to prove the well-posedness by hand. In the follwoing we will assume that $n=1$; the proof for $n>1$ follows similarly. 
Following the seminal method of Ebin and Marsden~\cite{ebin1970groups} we will show that the geodesic spray, as derived in Lemma~\ref{thm:geodesicequation}, extends to a smooth vector field on the Sobolev completion, which will allow us to apply Cauchy's theorem to conclude the local well-posedness of the equation.
To this end, we need the following statement regarding the invertibility of the operator $A_c$ under Neumann boundary conditions:

\noindent\emph{Claim: Let $f\in H^r_{\I^q}([0,2\pi],T\N)$ and $c\in \I^q([0,2\pi],\N)$  with $q-2\geq r\geq 0$ and $\pi\circ f=c$. Then the boundary value problem
\begin{equation}
A_c u(\theta) =f(\theta)\;, \qquad \nabla_{\pl_\theta} u(0)=u_0,\; \nabla_{\pl_\theta} u(2\pi)=u_1
\end{equation}
has a unique solution $u \in H^{r+2}_{\I^q}([0,2\pi],T\N)$, with $\pi\circ u = c$.}

Note that by subtracting any $H^q$ section that satisfy the boundary conditions, we can assume that the boundary conditions are homogeneous.
Then, a weak form of this equation is simply $G_c(u,w) = \int_0^{2\pi} g(f,w)\,\ud\theta$ for every $w\in H^1([0,2\pi],c^*T\N)$.
Since $c$ is fixed, $G_c$ is equivalent to the standard $H^1$ norm on $H^1([0,2\pi],c^*T\N)$.
By the Lax-Milgram theorem, there exists a unique solution $u_*\in H^1$.
We can then consider the equation $A_c u(\theta) =f(\theta)$ with initial conditions $u(0) = u_*(0)$, $\nabla_{\pl_\theta} u(0)=0$.
By moving to the weak form again, it follows that the solution for this initial value problem must be $u_*$, and so its regularity follows from standard initial-value ODE theory. This completes the proof of this claim.

To apply this theorem to the geodesic equation we need to observe that for any fixed time $t$ the boundary terms of the geodesic equation can be rewritten to yield 
exactly Neumann conditions for the system
\begin{multline*}
A_c(\nabla_{\partial_t}  c_t)=\bigg(-(\nabla_{\partial_t} A_c)c_t-g(v,\nabla_{\pl_s} c_t) A_c c_t-\frac12\Psi_c(c_t,c_t)\nabla_{\pl_s} v\\-g(\nabla_{\pl_s} c_t,A_c c_t) v 
		+ a_1(\ell_c)\mathcal R(c_t,\nabla_{\pl_s} c_t)v\bigg),\;\\		
\end{multline*}
where $(\nabla_{\partial_t} A_c)=\nabla_{\pl_t} \circ A_c-A_c \circ \nabla_{\pl_t}$, which is an operator of order 2.
In addition we have the boundary conditions
\begin{align*}
\nabla_{\theta}\nabla_{\partial_t} c_t\bigg|_{\theta = 0}&=F_0(c,c_t)\in \mathbb R\\
\nabla_{\theta}\nabla_{\partial_t} c_t\bigg|_{\theta = 2\pi}&=F_1(c,c_t)\in \mathbb R\;.
\end{align*}
where $F_0$ and $F_1$ can be calculated by applying the product formula for differentiation and the formula for swapping covariant derivatives to the boundary conditions in Lemma~\ref{thm:geodesicequation}. 

Thus by the claim above we can invert $A_c$ to rewrite the geodesic equation as 
\begin{multline*}
\nabla_{\partial_t}  c_t=A_c^{-1}\bigg(-(\nabla_{\partial_t} A_c)c_t-g(v,\nabla_{\pl_s} c_t) A_c c_t-\frac12\Psi_c(c_t,c_t)\nabla_{\pl_s} v\\-g(\nabla_{\pl_s} c_t,A_c c_t) v 
		+ a_1(\ell_c)\mathcal R(c_t,\nabla_{\pl_s} c_t)v\bigg)\;.
\end{multline*}
The right hand side of this equation defines a smooth mapping $$\Phi: T\mathcal I^q(D,\N)\to T\mathcal I^q(D,\N),$$ 
where the smoothness of $\Phi$ follows directly by counting derivatives, using the Sobolev embedding theorem and the result that $A_c$ and thus also $(\nabla_{\partial_t} A_c)$ and $A_c^{-1}$ are smooth. 
Thus we have interpreted the geodesic equation as an ODE (in $t$) on a Banach space of functions.
From here the proof of item 1 of Theorem~\ref{local_wellposedness} follows directly as in \cite[Theorem 4.4]{bauer2020fractional} and reduces to an application of the Cauchy theorem and the equivalence of fiber-wise quadratic smooth mappings $\Ph\colon  T\mathcal I^q(D,\N)\to  T\mathcal I^q(D,\N)$ and smooth sprays $S\colon  T\mathcal I^q(D,\N)\to  TT\mathcal I^q(D,\N)$.

To prove item 2 of Theorem~\ref{local_wellposedness}, we use Lemma~\ref{lem:no-loss-no-gain}, for $F$ the exponential map $G$ on $\I^q(D,\N)$, and $V\subset H^q_{\I^q}(D,T\N)$ a neighborhood of the zero section on which the exponential map is defined. 
It follows that the domain of existence of the geodesic equation (in $t$) and the neighborhoods for the exponential mapping are uniform in the Sobolev exponential $l\in \N$ and thus the result continues to hold on $\I^{q+l}_{\textup{loc}}(D,\N)$ and therefore also locally in the smooth category. 
\end{proof}


\section{Holonomy estimates: proof of Lemma~\ref{lem:high_order_ineq_per}}
\label{app:estimates}

We now prove the Sobolev estimates for manifolds-valued curves as stated in Lemma~\ref{lem:high_order_ineq_per}.
We start be proving some geometric estimates, culminating in bounds on the holonomy along a closed curve (Proposition~\ref{prop:estimate_hol}). 
The settings for the geometric estimates is as follows:

Let $(\N,g)$ be a complete Riemannian manifold of finite dimension, with bounded sectional curvature, $|K| \le K_\N$ and positive injectivity radius $\inj_\N>0$.
We denote by $\mathcal{R}$ the Riemann curvature of $g$.

Let $c:[0,a] \to \N$ be a curve, and let $V$ be a vector field along $c$.
Let $\Pi_{\theta_1}^{\theta_2}:T_{c(\theta_1)}\N \to T_{c(\theta_2)} \N$ be the parallel transport operator along $c$, and $\frac{D}{d\theta}$ the covariant derivative along $c$.
\begin{lem}\label{lem:fund_estimate_par_trans}
\[
\Abs{V(a) - \Pi_{0}^a V(0)} \le \int_0^a \Abs{ \frac{D}{d\theta} V(\theta)}\,\ud\theta.
\]
\end{lem}

\begin{proof}
Define $f(\theta) = \Pi_\theta^a V(\theta) - \Pi_0^a V(0)$.
Our goal is to bound $|f(a)|$.
Note that $f(0) = 0$, and that
\[
\frac{\pl}{\pl\theta} f(\theta)  = \Pi_\theta^a \frac{D}{d\theta} V(\theta).
\]
Therefore, using the fact that the parallel transport is an isometry, we have
\[
|f(a)| = \Abs{\int_0^a \frac{\pl}{\pl\theta} f(\theta) \,\ud\theta} \le \int_0^a \Abs{ \frac{D}{d\theta} V(\theta)} \,\ud\theta.
\]
\end{proof}

Let $c:[0,a]\to \N$ be a closed curve, $c(0)=c(a)=p$, with $\ell_c<2\inj_\N$.
Define a map $c(\theta,t): [0,a]\times [0,1]\to \N$, such that $c(\theta,\cdot)$ is the unique geodesic connecting $p$ and $c(\theta)$.
This is well defined since $\ell_c<2\inj_\N$ implies that $\dist(p,c(\theta)) < \ell_c/2 < \inj_\N$ for any $\theta$.
In other words, if we define $\gamma(\theta) = \exp_p^{-1}(c(\theta))$, then $c(\theta,t) := \exp_p(t \gamma(\theta))$.
For every $t_0\in [0,1]$, $c^{t_0} := c(\cdot,t_0):[0,a] \to \N$ is a closed curve based in $p$, and for $t_0=0$ it is the constant curve.

\begin{lem}\label{lem:length_curve}
There exists a constant $C_1$, depending only on $\inj_\N$ and the upper bound for the sectional curvature of $\N$, such that if the curve $c$ satisfies $\ell_c<C_1$, then $\ell_{c^t}\le \ell_c$ for every $t\in [0,1]$.
\end{lem}

\begin{proof}
In the following we will assume that $\ell_c<2\inj_\N$, otherwise the family $c^t$ is not well-defined.

It is obviously sufficient to prove that $|\pl_\theta c(\theta,t)| \le |\pl_\theta c(\theta,1)|$ for every $\theta$ and $t$.
Note that for a fixed $\theta_0$, $J(t) := \pl_\theta c(\theta_0,t)$ is a Jacobi field, hence it satisfies the Jacobi equation
\[
\frac{D^2}{dt^2} J + \mathcal{R}\brk{J, \pl_t c(\theta_0,t)} \pl_t c(\theta_0,t) = 0,
\]
with the initial conditions
\[
J(0) = 0, \qquad \frac{D}{dt}J(0) = \left. \frac{\pl}{\pl\theta}\right|_{\theta= \theta_0} \exp_p^{-1} c(\theta) =: \gamma'(\theta_0). 
\]
These initial conditions follow from the fact that
\begin{align*}
J(0) &= \left. \frac{\pl}{\pl \theta}\right|_{\theta= \theta_0} c(\theta,0) = \left. \frac{\pl}{\pl \theta}\right|_{\theta= \theta_0} p = 0,
\\
\frac{D}{dt}J(0) 
	&= \left. \frac{D}{\pl t} \frac{\pl}{\pl \theta} c\right|_{(\theta,t)= (\theta_0,0)} = \left. \frac{D}{\pl \theta} \frac{\pl}{\pl t} c \right|_{(\theta,t)= (\theta_0,0)}
\\&
	= \left. \frac{D}{\pl \theta} d_{t\gamma(\theta)}\exp_p[\gamma(\theta)]\right|_{(\theta,t)= (\theta_0,0)}
= \left. \frac{D}{\pl \theta} d_0 \exp_p[\gamma(\theta)]\right|_{\theta= \theta_0} 
\\&
=  \left. \frac{D}{\pl \theta} \gamma(\theta)\right|_{\theta= \theta_0} = \gamma'(\theta_0),
\end{align*}
where we used the fact that $d_0\exp_p = \id_{T_p\N}$, and that when $t=0$, $c(\theta,0) = p$ for all $\theta$, hence covariant derivative along $\theta$ is the same as the regular derivative in $\id_{T_p\N}$.

Note that we can always reparametrize $\theta$ such that $|\gamma'(\theta)| = 1$ for any $\theta$, hence $\Abs{\frac{D}{dt} J(0)} = 1$.

Our aim is to prove that $|J(t)| \le |J(1)|$.
The proof mimics the proof of Rauch's comparison theorem.
Define $f(t) := |J(t)|$; we want to prove that $\dot f(t) \ge 0$ for $t\in(0,1)$.
For brevity, write $\dot J := \frac{D}{d t}J$, $\ddot J := \frac{D^2}{dt^2} J$.
We then have
\[
\dot f = \frac{g(J,\dot J)}{|J|^2}.
\]
We have $J(0)=0$ and therefore, by the Jacobi equations, also $\ddot J(0) = 0$.
We therefore obtain that
\[
\dot J(t) = \dot J(0) + O(t^2), \qquad J(t) = t\dot J(0) + O(t^3),
\]
hence
\[
\dot f(t) = \frac{1}{t} + O(t).
\]
Using the Jacobi equations and the upper bound $K$ on the sectional curvature of $\N$, we obtain
\[
\begin{split}
\ddot f &= \frac{\brk{|\dot J|^2 + g(J,\ddot J)}|J|^2 - 2g(J,\dot J)^2}{|J|^4} \\
	&\ge \frac{g(J,\ddot J)}{|J|^2} - \frac{g(J,\dot J)^2}{|J|^4} = \frac{g(J,\ddot J)}{|J|^2} - \dot f^2  \\
	&=-\frac{g(J,\mathcal{R}(J,\pl_t c)\pl_t c)}{|J|^2} - \dot f^2 \\
	&\ge -\frac{K|\pl_t c|^2 |J|^2)}{|J|^2} - \dot f^2 = -K|\pl_t c|^2 - \dot f^2 \\
	&\ge -K\dist_\N^2(p, c(\theta_0)) - \dot f^2 \ge -\frac{K\ell_c^2}{4} - \dot f^2,
\end{split}
\]
where we used the fact that $|\pl_t c(\theta_0,t)|=\dist_\N^2(p, c(\theta_0))$ since $c(\theta_0,t)$ is a constant speed geodesic from $p$ to $c(\theta_0)$.
We obtain that
\[
\ddot f + \dot f^2 \ge - \frac{K\ell_c^2}{4}, \qquad \dot f(t) = \frac{1}{t} + O(t).
\]
From the Riccati comparison estimate \cite[Corollary~6.4.2]{petersen2016riemannian}, it follows that for $t>0$ we have
\[
\dot f(t) \ge
\begin{cases}
\frac{\sqrt{K}\ell_c}{2}\cot\brk{\frac{\sqrt{K}\ell_c}{2} t} 	& K>0, \, t\le \frac{2\pi}{\sqrt{K}\ell_c} \\
t 											& K = 0 \\
\frac{\sqrt{-K}\ell_c}{2}\coth\brk{\frac{\sqrt{-K}\ell_c}{2} t}	& K<0.
\end{cases}
\]
If $K\le 0$, it follows that $\dot f(t)>0$ for any $t>0$, and we are done.
If $K>0$, then by choosing $\ell_c < \pi/\sqrt{K}$, we obtain that $\dot f(t)$ is larger than a function that is positive in $(0,1]$. 
\end{proof}

We now state the main geometric estimate we need. 
Recall that, in two dimensions, the holonomy of a small closed curve is roughly the area enclosed by the curve times the curvature inside it, and that by the isoperimetric inequality, the area grows at most like the length of the curve squared.
The following proposition combines these statements (in any dimension) into a quantitative estimate on the holonomy:
\begin{prop}\label{prop:estimate_hol}
There exists a constant $C=C(K_\N,\inj_\N, \dim \N)>0$, such that for every closed curve $c\subset \N$ based in $T_p\N$,
\[
\Abs{\Hol_c - \id_{T_p\N}} \le \min\BRK{C\ell_c^2 \,,\, 2\sqrt{ \dim \N}}
\]
where $\Hol_c$ is the holonomy along $c$ and $\ell_c$ is the length of the curve $c$.
\end{prop}

\begin{proof}
Since $\Hol_c$ is an isometry of $T_p\N$, $|\Hol_c| = |\id_{T_p\N}| = \sqrt{\dim \N}$.
Therefore, by triangle inequality, we have $\Abs{\Hol_c - \id_{T_p\N}} \le 2\sqrt{\dim \N}$.

In the following, we assume that $C\ge 2\sqrt{\dim \N}/C_1^2$, where $C_1$ is defined in Lemma~\ref{lem:length_curve}, and therefore it is sufficient to prove that $\Abs{\Hol_c - \id_{T_p\N}}\le C \ell_c^2$ under the assumption that $\ell_c\le C_1$.

Fix a unit vector $v\in T_p\N$.
Our goal is to prove that $\Abs{\Hol_c v -v}\le C \ell_c^2$.
Define the family of curves $c(\theta,t) = c^t(\theta): [0,a]\times [0,1] \to \N$ as in Lemma~\ref{lem:length_curve}.
Define a vector field $X\in \Gamma(c^*T\N)$ by 
\[
X(\theta,t) := \Pi_{p}^{c^t(\theta)} v,
\]
where $\Pi_{p}^{c^t(\theta)}:T_p\N \to T_{c^t(\theta)} \N$ is the parallel transport along the curve $c^t$.
We have
\[
X(\theta, 0) = v, \quad X(0,t) = v, \quad X(a,t) = \Hol_{c^t} v.
\]
Since $c(a,t) = p$, the parallel transport along the curve $c(a,\cdot)$ is the identity, and so, by Lemma~\ref{lem:fund_estimate_par_trans} we have that
\[
\Abs{\Hol_c v -v} = \Abs{X(a,1) - X(a,0)} \le \int_0^1 \Abs{\frac{D}{\pl t} X(a,t)}\,\ud t.
\]
Since $c(0,t) = p$ for all $t$, the covariant derivative $\frac{D}{\pl t}$ along $(0,t)$ is simply the standard derivative $\frac{\pl}{\pl t}$.
Therefore, since $X(0,t)=v$ does not depend on $t$, we have $\frac{D}{\pl t} X(0,t) = 0$.
Hence, using Lemma~\ref{lem:fund_estimate_par_trans} again, we have
\[
\Abs{\frac{D}{\pl t} X(a,t)} \le \int_0^a \Abs{\frac{D}{\pl \theta}\frac{D}{\pl t} X(\theta,t)}\,\ud\theta.
\]

Since $X(\theta,t)$ is the parallel transport of $X(0,t) = v$ along the constant $t$ curve, we have $\frac{D}{\pl \theta} X(\theta,t) = 0$, and therefore $\frac{D}{\pl t}\frac{D}{\pl \theta} X(\theta,t) = 0$.
Combining this with
\[
\frac{D}{\pl \theta}\frac{D}{\pl t} X - \frac{D}{\pl t}\frac{D}{\pl \theta} X = \mathcal{R}\brk{\frac{\pl}{\pl \theta}\frac{\pl}{\pl t}} X
\]
(see, e.g., \cite[Chapter~4, Lemma~4.1]{carmo1992riemannian}),
we have
\[
\Abs{\frac{D}{\pl \theta}\frac{D}{\pl t} X} = \Abs{\mathcal{R}\brk{\frac{\pl c}{\pl \theta}\frac{\pl c}{\pl t}} X }
	\le K_\N \Abs{\frac{\pl}{\pl \theta}}\,\Abs{\frac{\pl}{\pl t}},
\]
where we used the fact that $|X| = |v| = 1$ since the parallel transport is an isometry.
Since $c(\theta,\cdot)$ is a constant speed geodesic from $p$ to $c(\theta) = c(\theta,1)$, and that $\dist(p,c(\theta)) \le \ell_c/2$, we have that 
\[
\Abs{\frac{\pl c}{\pl t}} \le \ell_c/2.
\]
Combining these estimates, we obtain
\[
\begin{split}
\Abs{\Hol_c v -v} 
	&\le \int_0^1 \Abs{\frac{D}{\pl t} X(a,t)}\,\ud t \le \int_0^1 \int_0^a \Abs{\frac{D}{\pl \theta}\frac{D}{\pl t} X}\,\ud\theta \,\ud t \\
	&\le K_\N \frac{\ell_c}{2} \int_0^1 \int_0^a \Abs{\frac{\pl c}{\pl \theta}} \,\ud\theta\,\ud t
		= K_\N \frac{\ell_c}{2} \int_0^1 \ell_{c^t} \,\ud t \le \frac{K_\N}{2} \ell_c^2,
\end{split}
\]
where in the last inequality we used Lemma~\ref{lem:length_curve} to estimate $\ell_{c^t}$.
\end{proof}

Using these holonomy estimates, we can now prove Lemma~\ref{lem:high_order_ineq_per}:

\noindent\textbf{Proof of Lemma~\ref{lem:high_order_ineq_per}.}
As mentioned at the beginning of Section~\ref{sec:estimates}, although $h(0) = h(2\pi)$ when $D=S^1$, it is not true that $H(0)=H(2\pi)$, where 
\[
H(\theta) = \Pi_\theta^0 h(\theta),
\]
because of holonomy effects.
Therefore, in order to prove \eqref{eq:higher_order_ineq_per} we cannot use Sobolev inequalities for  periodic functions verbatim, but rather use the result of Proposition~\ref{prop:estimate_hol}, which implies that for short curves $H$ is ``almost" periodic since the holonomy is small.
We will do so by induction over $k$ and $n$.

\textbf{Base step: the case $k=1$, $n=2$.}
Assume that $k=1$ and $n=2$.
When $\ell_c\ge 1$, the inequality \eqref{eq:higher_order_ineq} implies \eqref{eq:higher_order_ineq_per} by taking $a=1$.
We are left with the case $\ell_c<1$.

Recall that we denote by $\Pi_{\theta_1}^{\theta_2}$ the parallel transport from $T_{c(\theta_1)}\N$ to $T_{c(\theta_2)}\N$ along $c$ (in the direction dictated by the parameter $\theta$).
Now, by applying \eqref{eq:nabla_Pi} for $\nabla_{\pl_s} h$ and using the fundamental theorem of calculus, we have:
\[
\Pi_{\theta}^0 \nabla_{\pl_s} h(\theta) - \nabla_{\pl_s} h(0) = \int_0^\theta \frac{d}{d\sigma} \Pi_\sigma^0 \nabla_{\pl_s} h(\sigma)\, \ud\sigma = \int_0^\theta \Pi_\sigma^0 (\nabla_{\partial_\theta}\nabla_{\pl_s} h (\sigma))\, \ud\sigma.
\]
Integrating over $\theta$ with respect to $\ud s$, we obtain
\[
\nabla_{\pl_s} h(0) - \frac{1}{\ell_c} \int_{S^1} \Pi_{\theta}^0 \nabla_{\pl_s} h(\theta) \,\ud s(\theta) 
= -\frac{1}{\ell_c}\int_{S^1} \int_0^\theta \Pi_\sigma^0 (\nabla_{\partial_\theta} \nabla_{\pl_s} h (\sigma))\, \ud\sigma \,\ud s(\theta).
\]
Using again \eqref{eq:nabla_Pi}, we have 
\begin{multline*}
\int_{S^1} \Pi_\theta^0 \nabla_{\pl_s} h(\theta) \,\ud s(\theta) 
	= \int_0^{2\pi} \Pi_\theta^0 \nabla_{\partial_\theta} h(\theta)\,\ud\theta 
\\
	= \int_0^{2\pi} \frac{d}{d\theta}\Pi_\theta^0 h(\theta)\,\ud\theta = \Pi_{2\pi}^0 h(0) - h(0),
\end{multline*}
which is not necessarily zero since there the holonomy along $c$ might be non-trivial.
We therefore obtain
\begin{multline}\label{eq:induction_base_aux_1}
\nabla_{\pl_s} h(0) - \frac{1}{\ell_c} \brk{\Pi_{2\pi}^0 h(0) - h(0)} 
\\
= -\frac{1}{\ell_c}\int_{S^1} \int_0^\theta \Pi_\sigma^0 (\nabla_{\partial_\theta} \nabla_{\pl_s} h (\sigma))\, \ud\sigma \,\ud s(\theta).
\end{multline}
Similarly,
\begin{multline*}
 \nabla_{\pl_s} h(2\pi) - \Pi_{\theta}^{2\pi} \nabla_{\pl_s} h(\theta) = \int_\theta^{2\pi} \frac{d}{d\sigma} \Pi_\sigma^{2\pi} \nabla_{\pl_s} h(\sigma)\, \ud\sigma 
 \\
 = \int_\theta^{2\pi} \Pi_\sigma^{2\pi} (\nabla_{\partial_\theta}\nabla_{\pl_s} h (\sigma))\, \ud\sigma,
\end{multline*}
and
\[
\int_{S^1} \Pi_\theta^{2\pi} \nabla_{\pl_s} h(\theta) \,\ud s(\theta) 
	= \int_0^{2\pi} \frac{d}{d\theta}\Pi_\theta^{2\pi} h(\theta)\,\ud\theta = h(2\pi) - \Pi_0^{2\pi} h(2\pi).
\]
Thus, using the fact that $h(0)= h(2\pi)$ and $\nabla_{\pl_s} h(0) = \nabla_{\pl_s} h(2\pi)$, we have
\begin{multline}\label{eq:induction_base_aux_2}
\nabla_{\pl_s} h(0) - \frac{1}{\ell_c} \brk{h(0) - \Pi_{0}^{2\pi} h(0)}  
\\
= \frac{1}{\ell_c}\int_{S^1} \int_\theta^{2\pi} \Pi_\sigma^{2\pi} (\nabla_{\partial_\theta} \nabla_{\pl_s} h (\sigma))\, \ud\sigma \,\ud s(\theta).
\end{multline}
Adding \eqref{eq:induction_base_aux_1} and \eqref{eq:induction_base_aux_2}, we obtain
\[
\begin{split}
&\nabla_{\pl_s} h(0) - \frac{1}{\ell_c} \brk{\Pi_{2\pi}^0 h(0) - \Pi_{0}^{2\pi} h(0)}\\
	& = \frac{1}{2\ell_c}\int_{S^1} \brk{ \int_\theta^{2\pi} \Pi_\sigma^{2\pi} (\nabla_{\partial_\theta} \nabla_{\pl_s} h (\sigma))\, \ud\sigma - \int_0^\theta \Pi_\sigma^0 (\nabla_{\partial_\theta} \nabla_{\pl_s} h (\sigma)) \ud\sigma}\,\ud s(\theta).
\end{split}
\]
Therefore, using the fact that $\Pi_{\theta_1}^{\theta_2}$ is an isometry, we obtain that
\[
\begin{split}
|\nabla_{\pl_s} h(0)| &\le \frac{\Abs{\Pi_{2\pi}^0 - \Pi_{0}^{2\pi}}}{\ell_c} |h(0)| + \frac{1}{2\ell_c}\int_{S^1} \int_0^{2\pi} \Abs{\nabla_{\partial_\theta} \nabla_{\pl_s} h (\sigma)} \ud\sigma  \,\ud s(\theta) \\
	&= \frac{\Abs{\Pi_{2\pi}^0 - \Pi_{0}^{2\pi}}}{\ell_c} |h(0)| + \frac{1}{2} \int_{S^1} \Abs{\nabla_{\pl_s}^2 h (\sigma)} \ud s(\sigma).
\end{split}
\]
Using the estimate on the magnitude of the holonomy in Proposition~\ref{prop:estimate_hol}, we have
\[
|\nabla_{\pl_s} h(0)| \le \min\BRK{C\ell_c, \frac{2\sqrt{\dim \N}}{\ell_c}} |h(0)| + \frac{1}{2} \int_{S^1} \Abs{\nabla_{\pl_s}^2 h (\sigma)} \ud s,
\]
for some $C>0$ that depends only on the injectivity radius and on the bounds on the sectional curvature of $\N$.
In this inequality the point $0$ is arbitrary, hence the above holds for $h(\theta), \nabla_{\pl_s}h(\theta)$ instead of $h(0), \nabla_{\pl_s}h(0)$.
Squaring this inequality, and using the inequality $(a+b)^2 \le 2(a^2 + b^2)$ and Cauchy-Schwartz (or Jensen's) inequality, we obtain, for every $\theta$,
\begin{equation}\label{eq:aux_nabla_sh}
\begin{split}
|\nabla_{\pl_s} h(\theta)|^2 
	&\le \min\BRK{2C^2\ell_c^2, \frac{8\dim \N}{\ell_c^2}}|h(\theta)|^2 +  \frac{\ell_c}{2}\|\nabla_{\pl_s}^2h\|_{L^2(\ud s)}^2 \\
	&\le C'\min\BRK{\ell_c^2, \frac{1}{\ell_c^2}}|h(\theta)|^2 +  \frac{\ell_c}{2}\|\nabla_{\pl_s}^2h\|_{L^2(\ud s)}^2.
\end{split}
\end{equation}
Since we assumed $\ell_c\le 1$, Inequality \eqref{eq:aux_nabla_sh} implies \eqref{eq:higher_order_ineq_per} by integrating with respect to $\ud s$.

\textbf{Induction step.}
Now assume we have \eqref{eq:higher_order_ineq_per} for $n=2,\ldots,m$ and $k=1,\ldots, n-1$; we will now prove it for $n=m+1$, $k=1,\ldots, m$.
Denote the constant in \eqref{eq:higher_order_ineq_per} by $C_{k,n}$.
Besides $k$ and $n$, $C_{k,n}$ will depend also on the properties of the manifold $\N$ as stated in the formulation of the lemma, but we omit this dependence as it is fixed throughout the induction.

First, assume $k=1$.
If $\ell_c \ge \min\BRK{1,\brk{2C_{m-1,m}C_{1,m}}^{-1/2}}$, then \eqref{eq:higher_order_ineq} implies \eqref{eq:higher_order_ineq_per} for $k=1, n=m+1$, by letting $a=\min\BRK{1,\brk{2C_{m-1,m}C_{1,m}}^{-1/2}}$.
If $\ell_c \le \min\BRK{1,\brk{2C_{m-1,m}C_{1,m}}^{-1/2}}$, we have
\begin{multline*}
\| \nabla_{\pl_s} h \|^2_{L^2(\ud s)} 
	\le C_{1,m}\ell_c^2\brk{\| h\|^2_{L^2(\ud s)} + \| \nabla_{\pl_s}^m h \|^2_{L^2(\ud s)}} \\
	\le C_{1,m}\ell_c^2\brk{\| h\|^2_{L^2(\ud s)} + C_{m-1,m}\brk{\| \nabla_{\pl_s} h \|^2_{L^2(\ud s)}+\| \nabla_{\pl_s}^{m+1} h \|^2_{L^2(\ud s)}}},
\end{multline*}
where in the second line we applied the induction hypothesis to $\nabla_{\pl_s} h$.
Moving the $C_{1,m}C_{m-1,m}\ell_c^2 \| \nabla_{\pl_s} h \|^2_{L^2(\ud s)}$ to the other side, and noting that $C_{1,m}C_{m-1,m}\ell_c^2\le 1/2$ by assumption, we obtain that
\[
\| \nabla_{\pl_s} h \|^2_{L^2(\ud s)} \le 2C_{1,m}\ell_c^2\brk{\| h\|^2_{L^2(\ud s)} + C_{m-1,m}\| \nabla_{\pl_s}^{m+1} h \|^2_{L^2(\ud s)}},
\]
which completes the proof for $k=1$.

We now assume $k>1$.
If $\ell_c \ge \min\BRK{1,\brk{2C_{k-1,m}C_{1,k}}^{-1/2}}$, then \eqref{eq:higher_order_ineq} implies \eqref{eq:higher_order_ineq_per} for $k=1, n=m+1$, by letting $a=\min\BRK{1,\brk{2C_{k-1,m}C_{1,k}}^{-1/2}}$.
If $\ell_c \le \min\BRK{1,\brk{2C_{k-1,m}C_{1,k}}^{-1/2}}$, we have (by applying the induction hypothesis for $\nabla_{\pl_s}h$),
\begin{multline*}
\| \nabla_{\pl_s}^k h \|^2_{L^2(\ud s)} 
	\le C_{k-1,m}\ell_c^2\brk{\| \nabla_{\pl_s} h\|^2_{L^2(\ud s)} + \| \nabla_{\pl_s}^{m+1} h \|^2_{L^2(\ud s)}} \\
	\le C_{k-1,m}\ell_c^2\brk{C_{1,k}\brk{\| h\|^2_{L^2(\ud s)} + \| \nabla_{\pl_s}^k h \|^2_{L^2(\ud s)}}+\| \nabla_{\pl_s}^{m+1} h \|^2_{L^2(\ud s)}}.
\end{multline*}
Moving the $C_{k-1,m}C_{1,k}\ell_c^2 \| \nabla_{\pl_s}^k h \|^2_{L^2(\ud s)}$ to the other side, and noting that $C_{k-1,m}C_{1,k}\ell_c^2\le 1/2$ by assumption, we obtain that
\[
\| \nabla_{\pl_s}^k h \|^2_{L^2(\ud s)} \le 2C_{k-1,m}\ell_c^2\brk{C_{1,k} \| h\|^2_{L^2(\ud s)} + \| \nabla_{\pl_s}^{m+1} h \|^2_{L^2(\ud s)}},
\]
which completes the proof for $k>1$.

\section{Proof of Lemma~\ref{lem:high_order_bounds}}\label{app:em:high_order_bounds}
First, we note that for a function $f\in L^2(D)$ we have, for every $c\in \I^n(D,\N)$,
\[
\|f\|_{L^2(\ud\theta)} \le |D|^{1/2} \|f\|_{L^\infty}
\]
hence boundedness on metric balls of $\|\nabla_{\pl_s}^k |c'|\|_{L^\infty}$  implies boundedness of $\|\nabla_{\pl_s}^k |c'|\|_{L^2(\ud\theta)}$.
Lemma~\ref{lem:bounds_length} implies that under the assumption \eqref{eq:G_vs_length_weighted_H_1}, the $L^2(\ud\theta)$ and $L^2(\ud s)$ norms are equivalent on metric balls, hence boundedness on metric balls of $\|\nabla_{\pl_s}^k |c'|\|_{L^\infty}$ also implies boundedness of $\|\nabla_{\pl_s}^k |c'|\|_{L^2(\ud s)}$.
Therefore, by Lemma~\ref{lem:local_bounded}, our goal is to show that
\[
\|D_{c,h} (\nabla_{\pl_s}^k|c'|) \|_{L^p} \le C(1 + \|\nabla_{\pl_s}^k|c'|\|_{L^p}) \|h\|_{G_c},
\]
where $p=\infty$ for $k=0,\ldots, n-2$ and $p=2$ for $k=n-1$.
We will first prove the case $p=\infty$ by induction on $k$, and then treat the case $p=2$, $k=n-1$ (in which the cases $L^2(\ud\theta)$ and $L^2(\ud s)$ are similar, so for brevity, we simply write $L^2$).

The claim for $k=0$ was proven in Lemma~\ref{lem:bound_speed}.
We now assume the claim is true up to $k-1$ and prove it for $k$.
First, note that
\[
\begin{split}
D_{c,h} (\nabla_{\pl_s}^k|c'|) 
	&= \nabla_{\pl_s}^k (g(v,\nabla_{\pl_s} h) |c'|) - \sum_{i=0}^{k-1} \binom{k}{i+1} \nabla_{\pl_s}^i g(v,\nabla_{\pl_s} h) \nabla_{\pl_s}^{k-i} |c'| \\
	&=\sum_{i=0}^{k} \brk{\binom{k}{i}-\binom{k}{i+1}} \nabla_{\pl_s}^i g(v,\nabla_{\pl_s} h) \nabla_{\pl_s}^{k-i} |c'|,
\end{split} 
\]
where we use the convention $\binom{k}{k+1}=0$.
This can be easily proved by induction using \eqref{eq:speed_derivative}.
From this it follows that 
\begin{multline}
\label{eq:D_ch_nabla_aux}
\Abs{D_{c,h} (\nabla_{\pl_s}^k|c'|)} 
	\lesssim \sum_{i=0}^{k} \Abs{\nabla_{\pl_s}^i g(v,\nabla_{\pl_s} h)} \Abs{\nabla_{\pl_s}^{k-i} |c'|}
\\
	\lesssim \sum_{i=0}^{k}\sum_{j=0}^i  \Abs{\nabla_{\pl_s}^j v} \Abs{\nabla_{\pl_s}^{i-j+1} h} \Abs{\nabla_{\pl_s}^{k-i} |c'|},
\end{multline}
where the constant depends only on the indices $i,j,k$.
Using the induction hypothesis, we obtain (using the fact that $|v| =1$),
\[
\Abs{D_{c,h} (\nabla_{\pl_s}^k|c'|)} 
	\lesssim \Abs{\nabla_{\pl_s} h} \Abs{\nabla_{\pl_s}^{k} |c'|}+  \sum_{i=0}^{k}\sum_{j=0}^i  \Abs{\nabla_{\pl_s}^j v} \Abs{\nabla_{\pl_s}^{i-j+1} h}
\]
on every metric ball.
Our assumption \eqref{eq:G_vs_L_infty_der} implies that for $i=1,\ldots,n-1$, we have $\|\nabla_{\pl_s}^{i} h\|_{L^\infty}\le C\|h\|_{G_c}$ on every metric ball.
Therefore, we obtain, as long as $k\le n-2$
\[
\Abs{D_{c,h} (\nabla_{\pl_s}^k|c'|)} 
	\lesssim \brk{\Abs{\nabla_{\pl_s}^{k} |c'|}+  \sum_{j=0}^{k}  \Abs{\nabla_{\pl_s}^j v}}\|h\|_{G_c}
\]
on every metric ball.

In order to complete the proof (for the $L^\infty$ case), we need to show that
\begin{align}
\label{eq:L_infty_bound_v}
\|\nabla_{\pl_s}^k v\|_{L^\infty} \qquad k=0,\ldots, n-2
\end{align}
is bounded on every metric ball.
The case $k=0$ is trivial, since $|v|=1$ by definition.
Note that
\[
D_{c,h}|\nabla_{\pl_s}^k v| = g\brk{\nabla_h \nabla_{\pl_s}^k v , \frac{\nabla_{\pl_s}^k v}{|\nabla_{\pl_s}^k v|}} \le |\nabla_h \nabla_{\pl_s}^k v|.
\]
Therefore, in order to use Lemma~\ref{lem:local_bounded} for the function $|\nabla_{\pl_s}^k v|$, we need to show that
\[
|\nabla_h \nabla_{\pl_s}^k v| \le C(1+ \|\nabla_{\pl_s}^k v\|_\infty) \|h\|_{G_c}
\]
on every metric ball.
Using \eqref{eq:noncommuting_cov_der}, we obtain
\[
\begin{split}
\nabla_h \nabla_{\pl_s}^k v
	&=\nabla_{\pl_s} \nabla_h \nabla_{\pl_s}^{k-1}v -g(v,\nabla_{\pl_s} h) \nabla_{\pl_s}^k v + \calR(v,h) \nabla_{\pl_s}^{k-1} v  \\
	&=\nabla_{\pl_s}^k \nabla_h v -\sum_{i=0}^{k-1} \nabla_{\pl_s}^i (g(v,\nabla_{\pl_s} h) \nabla_{\pl_s}^{k-i} v) + \sum_{i=0}^{k-1} \nabla_{\pl_s}^i (\calR(v,h) \nabla_{\pl_s}^{k-1-i} v) \\
	&=\nabla_{\pl_s}^{k+1} h -\sum_{i=0}^{k} \nabla_{\pl_s}^i (g(v,\nabla_{\pl_s} h) \nabla_{\pl_s}^{k-i} v) + \sum_{i=0}^{k-1} \nabla_{\pl_s}^i (\calR(v,h) \nabla_{\pl_s}^{k-1-i} v),
\end{split}
\]
where in the last line we used the fact that
\[
\nabla_h v = \nabla_{\pl_s} h - g(v,\nabla_{\pl_s} h)v,
\]
which follows immediately from \eqref{eq:speed_derivative}.
We therefore have, 
\[
\begin{split}
\nabla_h \nabla_{\pl_s}^k v
	 &= \nabla_{\pl_s}^{k+1} h 
		-\sum_{i=0}^{k}\sum_{j=0}^i\sum_{l=0}^j \binom{i}{j}\binom{j}{l} g(\nabla_{\pl_s}^l v,\nabla_{\pl_s}^{j-l+1} h) \nabla_{\pl_s}^{k-j} v \\
	&\quad + \sum_{i=0}^{k-1}\sum_{j=0}^i\sum_{l=0}^j\sum_{m=0}^l \binom{i}{j}\binom{j}{l}\binom{l}{m} \nabla_{\pl_s}^{j-l}\calR(\nabla_{\pl_s}^{m}v,\nabla_{\pl_s}^{l-m} h) \nabla_{\pl_s}^{k-1-j} v,
\end{split}
\]
where we repeatedly used
\begin{multline*}
\nabla_{\pl_s} \brk{\calR(X,Y) Z} = (\nabla_{\pl_s} \calR)(X,Y)Z + \calR(\nabla_{\pl_s} X,Y)Z 
\\
+ \calR(X,\nabla_{\pl_s} Y)Z + \calR(X,Y) \nabla_{\pl_s}Z.
\end{multline*}

Using the fact that $\nabla_{\pl_s}^r \calR$ is bounded for every $r$,\footnote{Note that by Lemma~\ref{lem:bounded_image}, the whole analysis here is done on a compact subset of $\N$ (the closure of the image of $B(c_0,r)$).
Hence the boundedness of $\calR$ and its covariant derivatives follows from the smoothness of $\N$, and does not require any global bounded geometry assumption on $\N$ (except from completeness).}
we obtain the bound
\begin{align}
|\nabla_h \nabla_{\pl_s}^k v| 
	&\lesssim |\nabla_{\pl_s}^{k+1} h|
	 	+ \sum_{j=0}^k\sum_{l=0}^j |\nabla_{\pl_s}^l v|\,|\nabla_{\pl_s}^{j-l+1} h|\,| \nabla_{\pl_s}^{k-j} v|\\&\qquad\qquad
		+\sum_{j=0}^{k-1}  |\nabla_{\pl_s}^{k-1-j} v| \sum_{l=0}^j\sum_{m=0}^l |\nabla_{\pl_s}^{m}v|\,|\nabla_{\pl_s}^{l-m} h| \\
	&\lesssim |\nabla_{\pl_s}^{k+1} h| + |\nabla_{\pl_s}^k v| |\nabla_{\pl_s} h| + \sum_{i=0}^k P_i |\nabla_{\pl_s}^i h|,
\end{align}
where $P_i$ are polynomials in $|\nabla_{\pl_s} v|,\ldots, |\nabla_{\pl_s}^{k-1}v|$.
By the induction hypothesis $\|\nabla_{\pl_s}^j v\|_\infty$ is bounded on metric balls for $j=0,\ldots, k-1$, hence $P_i$ is bounded on metric balls.
Using, this, and assumption \eqref{eq:G_vs_L_infty_der}, we obtain that, as long as $k\le n-2$,
\[
|\nabla_h \nabla_{\pl_s}^k v| \lesssim (1 + |\nabla_{\pl_s}^k v|)\|h\|_{G_c},
\]
which completes the proof of \eqref{eq:L_infty_bound_v} and hence of \eqref{eq:bound_speed_L_infty}.

It remains to prove \eqref{eq:bound_speed_L_2} for $k=n-1$, that is, to prove that
\[
\|D_{c,h} (\nabla_{\pl_s}^{n-1}|c'|) \|_{L^2} \le C(1 + \|\nabla_{\pl_s}^{n-1}|c'|\|_{L^2}) \|h\|_{G_c}.
\]
Using \eqref{eq:D_ch_nabla_aux} we have
\[
\begin{split}
&\Abs{D_{c,h} (\nabla_{\pl_s}^{n-1}|c'|)} 
	\lesssim \sum_{i=0}^{n-1}\sum_{j=0}^i  \Abs{\nabla_{\pl_s}^j v} \Abs{\nabla_{\pl_s}^{i-j+1} h} \Abs{\nabla_{\pl_s}^{n-1-i} |c'|} \\
	&\quad\lesssim \Abs{\nabla_{\pl_s}^{n} h}|c'| + \|h\|_{G_c}\sum_{i=0}^{n-1}\sum_{j=0}^i  \Abs{\nabla_{\pl_s}^j v} \Abs{\nabla_{\pl_s}^{n-1-i} |c'|} \\
	&\quad\lesssim \Abs{\nabla_{\pl_s}^{n} h}|c'| + \|h\|_{G_c}\brk{\Abs{\nabla_{\pl_s}^{n-1} |c'|} + |c'| \Abs{\nabla_{\pl_s}^{n-1} v} +\sum_{i,j=0}^{n-2} \Abs{\nabla_{\pl_s}^j v} \Abs{\nabla_{\pl_s}^{i} |c'|}} \\
	&\quad\lesssim \Abs{\nabla_{\pl_s}^{n} h} + \|h\|_{G_c}\brk{\Abs{\nabla_{\pl_s}^{n-1} |c'|} + \Abs{\nabla_{\pl_s}^{n-1} v} + 1}
\end{split}
\]
where in the second inequality we used \eqref{eq:G_vs_L_infty_der}, and in the bounds \eqref{eq:bound_speed_L_infty} and \eqref{eq:L_infty_bound_v} on metric balls.
Squaring this and integrating, we obtain, using \eqref{eq:G_vs_H_n} for the first term,
\[
\|D_{c,h} (\nabla_{\pl_s}^{n-1}|c'|) \|_{L^2} \le C(1 + \||\nabla_{\pl_s}^{n-1}v|\|_{L^2} + \|\nabla_{\pl_s}^{n-1}|c'|\|_{L^2}) \|h\|_{G_c}.
\]
Therefore, we are left to show that $ \||\nabla_{\pl_s}^{n-1}v|\|_{L^2} $ is bounded on metric balls.

As before, we need to show that
\begin{align}
\label{eq:nabla_h_bound}
\|\nabla_h \nabla_{\pl_s}^{n-1} v\|_{L^2} \le C(1+ \|\nabla_{\pl_s}^{n-1} v\|_{L^2}) \|h\|_{G_c},
\end{align}
and we have shown that
\[
\begin{split}
|\nabla_h \nabla_{\pl_s}^{n-1} v| 
	&\lesssim |\nabla_{\pl_s}^{n} h| + |\nabla_{\pl_s}^{n-1} v| |\nabla_{\pl_s} h| + \sum_{i=0}^{n-1} P_i |\nabla_{\pl_s}^i h|
\end{split}
\]
where $P_i$ are polynomials in $|\nabla_{\pl_s} v|,\ldots, |\nabla_{\pl_s}^{n-2}v|$, which are bounded on metric balls.
We therefore have, using \eqref{eq:G_vs_L_infty_der} that
\[
|\nabla_h \nabla_{\pl_s}^{n-1} v| 
	\lesssim |\nabla_{\pl_s}^{n} h| + \|h\|_{G_c} \brk{1+ |\nabla_{\pl_s}^{n-1} v|}.
\]
Squaring, integrating and using \eqref{eq:G_vs_H_n}, we obtain \eqref{eq:nabla_h_bound}, which completes the proof.


\bibliographystyle{abbrv}
\bibliography{refs}
\end{document}